\documentclass{article}
\usepackage[T1]{fontenc}
\usepackage{times}
\usepackage{rawfonts}
\usepackage{latexsym}
\usepackage{amsmath}
\usepackage{amsfonts}
\usepackage{amssymb}
\usepackage{enumerate}
\usepackage{xspace}
\usepackage{verbatim}
\usepackage{theorem}
\newcommand{\diag}{\text{diag}}
\newtheorem{theorem}{Theorem}[section]
\newtheorem{proposition}[theorem]{Proposition}
\newtheorem{lemma}[theorem]{Lemma}
\newtheorem{corollary}[theorem]{Corollary}
{\theorembodyfont{\rm} \newtheorem{assumption}[theorem]{Assumption} }
{\theorembodyfont{\rm} \newtheorem{example}[theorem]{Example} }
{\theorembodyfont{\rm} \newtheorem{remark}[theorem]{Remark} }
{\theorembodyfont{\rm} \newtheorem{algorithm}[theorem]{Algorithm} }
\newcommand{\proof}[1][]{{{\sc Proof}\xspace#1. }}

\newcommand{\refeq}[1]{\klasm{\ref{eq:#1}}}

\newcommand{\reza}{ \mathbb{R}\hspace{0.1mm}}

\newcommand{\varx}{x}

\newcommand{\vary}{y}
\newcommand{\klasm}[1]{(#1)}
\newcommand{\kla}[1]{(#1)}
\newcommand{\klafn}[1]{(#1)}

\newcommand{\klami}[1]{(#1)}
\newcommand{\myint}[2]{(I#1)(#2)}

\newcommand{\ints}[4]{\int_{#1}^{#2}#3 \, #4}

\newcommand{\cdott}{\hspace{0.4mm}}
\newcommand{\cdottsm}{\hspace{0.3mm}}
\newcommand{\for}{\quad \text{for} \ \ }
\newcommand{\foreach}{\quad \text{for each} \ \ }

\newcommand{\intervalarg}[3]{#2 \hspace{0mm} \le \hspace{0mm} #1 \hspace{0mm} \le \hspace{0mm} #3}
\newcommand{\intervalargno}[3]{#2 \le #1 \le #3}

\newcommand{\xmax}{b}
\newcommand{\xmin}{a}
\newcommand{\xminpl}{a+}
\newcommand{\diffxmaxxmin}{\xmax-\xmin}

\newcommand{\diffxmaxxminkla}{\kla{\xmax-\xmin}}
\newcommand{\interval}[2]{\eckkla{#1, #2 }}

\newcommand{\eq}{\ = \ }

\newcommand{\mywlog}{without loss of generality\xspace}

\newcommand{\mfrac}[2]{\dfrac{\mbox{\footnotesize \raisebox{-0.5mm}{$#1$}}}%
{\mbox{\footnotesize \raisebox{0.8mm}{$#2$}}}}
\newenvironment{mylist_indent}{%
\begin{list}{\tinybullet}
{\setlength{\topsep}{0.2cm}
\setlength{\itemsep}{0mm}
\setlength{\labelwidth}{2mm}
\setlength{\labelsep}{3mm}
\setlength{\itemindent}{0mm}
\setlength{\leftmargin}{5mm}
}}{\end{list}}

\newcommand{\alpinv}{\alpha^{{\scriptscriptstyle (-1)}}}
\newcommand{\betinv}{\beta^{{\scriptscriptstyle (-1)}}}
\newcommand{\gaminv}{\gamma^{{\scriptscriptstyle (-1)}}}

\newcommand{\mysum}[2]{\mathop{\mbox{\small $\dis
\sum\nolimits$}}_{#1}^{#2}}
\newcommand{\myomega}{w}
\newcommand{\tinybullet}{{\tiny \raisebox{0.6mm}{$ \bullet $}}}
\newcommand{\with}{\quadsm \text{with} \ \ }
\newcommand{\withsome}{\quad \text{with some}  \  \ }

\newcommand{\stepsize}{step size\xspace}

\newcommand{\myseqq}[3]{#1, \allowbreak #2, \allowbreak \ldots, #3}
\newcommand{\myseqqsh}[3]{#1, \allowbreak \ldots, #3}
\newcommand{\dis}{\displaystyle}
\newcommand{\bn}{\bigskip \noindent}

\newcommand{\ie}{i.\,e.,\xspace}
\newcommand{\nmax}{N}

\newcommand{\bs}{-}
\newcommand{\stepsizes}{\stepsize{}s\xspace}
\newcommand{\Stepsizes}{Step sizes\xspace}

\newcommand{\eg}{e.g.,\xspace}

\newcommand{\minus}{-}
\newcommand{\wrt}{with respect to\xspace}

\newcommand{\Landau}{\mathcal{O}}

\newcommand{\Landauno}[1]{\Landau\kla{#1}}
\newcommand{\Landaula}[1]{\Landau\klala{#1}}

\newcommand{\as}{\quad \text{as} \ \ }

\newcommand{\cf}{cf.~}

\newcommand{\Ftp}{For this purpose\xspace}

\newcommand{\ableit}[2]{#1^{\klafn{#2}}}

\newcommand{\myunderbrace}[2]{\underbrace{\raisebox{-0.8mm}{\vphantom{$ #1 $}} #1 }_{\dis #2}}

\newcommand{\lfrac}[2]{#1/ #2}
\newcommand{\myxi}{\xi}
\newcommand{\genfunc}{generating function\xspace}
\newcommand{\genfuncs}{\genfunc{}s\xspace}

\newcommand{\koza}{ \mathbb{C} }

\newcommand{\modul}[1]{\vert \hspace{0.3mm} #1 \hspace{0.3mm} \vert}

\newcommand{\plus}{+}
\newcommand{\maxnorm}[1]{\norm{#1}_\infty }

\newcommand{\norm}[1]{\Vert \hspace{0mm} #1 \hspace{0mm} \Vert}

\newcommand{\prim}[1]{#1^{\prime}}

\newcommand{\gridpoint}{node\xspace}
\newcommand{\gridpoints}{\gridpoint{}s\xspace}

\newcommand{\myoverbrace}[2]{\overbrace{\raisebox{0.8mm}{\vphantom{$ #1 $}} #1}^{\dis #2}}

\newcommand{\feldstretch}[1]{\renewcommand{\arraystretch}{#1}}

\newcommand{\myrnn}[1][\n]{\reza^{#1\times#1}}
\newcommand{\Landausm}[1]{\Landau\klasm{#1}}
\newcommand{\defeq}{:=}
\newenvironment{myenumerate}{%
\begin{list}{(\alph{enumcount})}
{\setcounter{enumcount}{1}\usecounter{enumcount}
\setlength{\topsep}{1mm}
\setlength{\itemsep}{0mm}
\setlength{\labelwidth}{0mm}
\setlength{\labelsep}{1mm}
\setlength{\itemindent}{1mm}
\setlength{\leftmargin}{0mm}
}}{\end{list}}

\newcommand{\klasmsh}[1]{\mbox{\scriptsize\raisebox{0.2mm}{$($}}\nolinebreak\hspace{-0.0mm}\raisebox{-0.1mm}{$#1$}\hspace{-1mm}\nolinebreak\mbox{\scriptsize\raisebox{0.2mm}{$)$}}}
\newcommand{\proofendspruch}[1][]{This completes the proof#1.\proofend}
\newcommand{\proofend}{\myhfill $ \endproof $}
\newcommand{\myhfill}{\hspace*{\fill}}
\renewcommand{\proofend}{\myhfill \endproof}
\def\endproof{$\Box$}
\newcommand{\rhs}{right\bs{}hand side\xspace}
\newcommand{\xn}[1][n]{x_{#1}}
\newcommand{\xm}{\xn[\emm]}

\newcommand{\xs}{\xn[s]}
\newcommand{\xr}{\xn[r]}

\newcommand{\myassump}[2]{\begin{assumption} #1 \label{th:#2} \end{assumption}}

\newcommand{\matvecform}{matrix\bs{}vector formulation\xspace}

\newenvironment{myenumerate_indent}{%
\begin{list}{(\alph{enumcount})}
{\setcounter{enumcount}{1}\usecounter{enumcount}
\setlength{\topsep}{1mm}
\setlength{\itemsep}{-0mm}
\setlength{\labelwidth}{5mm}
\setlength{\labelsep}{2mm}
\setlength{\itemindent}{-0mm}
\setlength{\leftmargin}{7mm}
}}{\end{list}}

\newcommand{\tee}{{ \hspace{-0.1mm} \top \! }}
\newcommand{\quadti}{\hspace{1.5mm}}

\newcommand{\Ah}{A_h}
\newcommand{\Fh}{G_h^\delta}
\newcommand{\Fha}{G_{h,1}}
\newcommand{\Fhb}{G_{h,2}^\delta}
\newcommand{\Bh}{B_h}
\newcommand{\Uh}{U_h}
\newcommand{\Uhinv}{U_h^{-1}}
\newcommand{\Vh}{V_h}
\newcommand{\Vhinv}{V_h^{-1}}
\newcommand{\Wh}{W_h}
\newcommand{\Whinv}{W_h^{-1}}

\newcommand{\Mhs}[1][t]{M_h}
\newcommand{\Chs}[1][t]{C_h}
\newcommand{\erra}{\mathcal{E}}
\newcommand{\Sh}{S_h}
\newcommand{\Shinv}{S_h^{-1}}

\newcommand{\myl}{\ell}
\newcommand{\klabi}[1]{\big(\cdottsm #1 \cdottsm\big)}
\newcommand{\eckklabi}[1]{\big[\cdott #1 \cdott \big]}

\newcommand{\quadsm}{\hspace{3mm}}

\newcommand{\klala}[1]{\Big(\cdottsm #1 \cdottsm\Big)}

\newcommand{\n}{n}

\newcommand{\fndelta}[1][n]{f_{#1}^\delta}

\newcommand{\enndelta}[1][s]{e_{#1}^\delta}

\newcommand{\undelta}[1][n]{u_{#1}^\delta}

\newcommand{\voltinteqspur}{Volterra integral equations\xspace}
\newcommand{\Landaubi}[1]{\Landau\klabi{#1}}
\newcommand{\Ehdelta}[1][n]{\Delta_h^\delta}
\newcommand{\Fhdelta}[1][n]{F_h^\delta}

\newcommand{\rhss}{\rhs{}s\xspace}

\newcommand{\mycitea}[2]{#1~\cite{#1[#2]}}
\newcommand{\mynocitea}[2]{\cite{#1[#2]}}
\newcommand{\myciteatwo}[2]{#1~\cite{#2}}
\newcommand{\mycitec}[4]{#1\cdott{\myslash}\cdott#2\cdott{\myslash}#3~\cite{#1_#2_#3[#4]}}
\newcommand{\myciteb}[3]{#1\cdott /\cdott #2~\cite{#1_#2[#3]}}
\newcommand{\mycitebtwo}[3]{#1\cdott /\cdott #2~\cite{#3}}
\newcommand{\myslash}{\cdottsm/\cdottsm\xspace}
\newcommand{\myk}[2]{k_{#1#2}}
\newcommand{\mykkom}[2]{k_{#1,#2}}

\newcommand{\psitkomb}[2]{g_h(\xn[#1],\xn[#2])}

\newcommand{\voltinteqs}{Volterra integral equations of the first kind %
with smooth kernels\xspace}

\newcommand{\voltinteqssh}{Volterra integral equations\xspace}

\newcommand{\msm}{multistep method\xspace}
\newcommand{\mss}{multistep scheme\xspace}
\newcommand{\msms}{\msm{}s\xspace}

\newcommand{\ksv}[1][\emm]{$ #1 $\bs{}step method\xspace}
\newcommand{\ksvs}[1][\emm]{$ #1 $\bs{}step methods\xspace}
\newcommand{\emm}{m}

\newcommand{\fxn}[1][n]{f\klasm{x_{#1}}}
\newcommand{\genpol}{\generating polynomial\xspace}

\newcommand{\fps}{power series\xspace}

\newcommand{\locerr}[1]{\eta\klasm{#1}}
\newcommand{\locerrbi}[1]{\eta\klabi{#1}}
\newcommand{\mycitectwo}[4]{#1\cdott{\myslash}\cdott#2\cdott{\myslash}#3~\cite{#4}}

\newcommand{\mysumtxt}[2]{\sum_{#1}^{#2}}
\newcommand{\intstxt}[4]{\int_{#1}^{#2}#3 \, #4}

\newcommand{\eckkla}[1]{[\cdott #1 \cdott ]}

\newcommand{\myfun}{\psi}

\newcommand{\myfuns}{\myfun_s}

\newcommand{\myphi}{\varphi}
\newcommand{\wstart}[1]{\myomega_{#1}}

\newcommand{\idstar}{\stackrel{(*)}{ = }}

\newcommand{\cont}{continuous\xspace}

\newcommand{\hdeltatonull}{\kla{h, \cdott \delta} \to 0}

\newcommand{\mydeltax}{h}

\newcommand{\N}{N}

\newcommand{\globerr}[1][\n]{r_h(\xn[#1])}

\newcommand{\rh}{R_{h}}
\newcommand{\en}[1][\n]{e_{\n}}

\newcommand{\remarkend}{\quad $ \vartriangle $}
\newcommand{\inset}[1]{\{ \, #1 \, \}}

\newcommand{\myr}{r}

\newcommand{\inverse}{inverse\xspace}

\newcommand{\myast}{\cdot}
\newcommand{\octave}{O\textsc{ctave}\xspace}

\newcommand{\mymu}{\mu}

\newcommand{\eset}{E}
\renewcommand{\max}{\mathop{\textup{max}}}
\newcommand{\mya}{a}
\newcommand{\myb}{b}
\newcommand{\myrho}{\varrho}

\newcommand{\mygamma}{\gamma}
\newcommand{\pol}{\mathcal{P}}
\newcommand{\polr}[1][r]{\mathcal{P}_{#1}}

\newcommand{\myp}{{p_0}}
\newcommand{\mym}{m}
\newcommand{\myd}{m}

\newcommand{\mygammainvn}[1][n]{\gaminv_{#1}}
\newcommand{\mygammainvnmu}[1][n]{\gaminv_{#1}}
\newcommand{\mygammainvnmupur}{\gaminv_{0}}
\newcommand{\mygamman}[1][n]{\mygamma_{#1}}
\newcommand{\mygammanmu}[1][n]{\mygamma_{#1}}
\newcommand{\mygammanmupur}{\mygamma_{0}}

\newcommand{\mtilde}{M}
\newcommand{\nmaxpmu}{\nmax_1}
\newcommand{\nmaxpmumo}{\nmax-\mym-\mu}
\newcommand{\myendelta}[1][n]{e_{#1}^\delta}
\newcommand{\generating}{characteristic\xspace}

\newcommand{\maxconsistenceorder}{maximal order\xspace}
\newcommand{\consistenceorder}{order\xspace}

\newcommand{\myq}{\tau}
\newcommand{\myqq}{p}

\newcommand{\myxpo}{\tau}

\newcommand{\myfunspectb}[1]{g_h(#1)}

\newcommand{\psidel}[1][s]{\myfun^\delta_{#1}}
\newcommand{\phidel}[1][r]{\myphi^\delta_{#1}}
\newcommand{\mykap}{\ell}
\newcommand{\inapp}{initial approximation\xspace}
\newcommand{\inapps}{\inapp{}s\xspace}
\newcommand{\cond}{\textup{cond}\xspace}
\newcommand{\Ns}[1][s]{N_{#1}}
\newcommand{\Nsmin}{\underline{N}}
\newcommand{\hs}[1][s]{h_{#1}}
\newcommand{\hstar}{h_{*}}
\newcommand{\hstst}{h^\prime}

\newcommand{\Ih}[1]{\Delta({#1})}
\newcommand{\Hdelta}{H^\delta}
\newcommand{\Hdeltatwo}{\tilde{H}^\delta}
\newcommand{\Hset}{\Sigma}
\newcommand{\Mdelta}{M^\delta}
\newcommand{\udeltas}[2]{u^\delta(#1,#2)}
\newcommand{\squer}{\overline{s}}

\newcommand{\hdelta}{h(\delta)}
\newcommand{\hdeltapl}{h^+(\delta)}

\newcommand{\myeta}{\beta}
\newcommand{\myhbound}{\overline{h}}
\newcommand{\muL}{\mu L}
\newcommand{\myL}{L}
\newcommand{\myC}{C}
\newcommand{\myChat}{\widehat{C}}

\newcommand{\Cpl}[2][\myL]{\myC^{#2}_{#1}}
\newcommand{\UCpl}[1]{\myChat^{#1}}
\newcommand{\Cp}[1]{C^{#1}}
\newcommand{\Nmin}{N_\textup{min}}
\newcommand{\hmax}{h_\textup{max}}

\newcommand{\Rtil}{\widetilde{R}}
\newcommand{\Deltap}{\Delta f}
\newcommand{\schurpolynomial}{Schur polynomial\xspace}
\DeclareMathOperator*{\mycup}{\cup}

\title{The regularizing properties of multistep methods for first kind Volterra integral equations with smooth kernels}
\author{Robert Plato%
\thanks{Department of Mathematics, University of Siegen,
Walter-Flex-Str.~3, 57068 Siegen, Germany.
}
}

\begin{document}

\maketitle
\newcounter{enumcount}
\renewcommand{\theenumcount}{(\alph{enumcount})}

\bibliographystyle{plain}

\begin{abstract}
We study quadrature methods for solving Volterra integral equations of the first kind  with smooth kernels under the presence of noise in the right-hand sides,
with the quadrature methods being generated by linear multistep methods.
The regularizing properties of an a priori choice of the \stepsize are analyzed, with
the smoothness of the involved functions carefully taken into consideration.
The balancing principle as an adaptive choice of the \stepsize is also studied.
It is considered in a version which sometimes requires less amount of computational work than the standard version of this principle.
Numerical results are included.
\end{abstract}
\section{Introduction}
\label{intro-multstep-methods}
In this paper we consider linear \voltinteqspur of the following form,
\begin{align}
\klasm{Au}\klasm{\varx} =
\ints{\xmin}{\varx}
{k\kla{\varx,\vary} u\klami{\vary} }
{d \vary}
= f\klasm{\varx}
\for \intervalarg{\varx}{\xmin}{\xmax},
\label{eq:volterra-inteq}
\end{align}
with a sufficiently smooth kernel function
$ k: \inset{(x,y) \in \reza^2 \ \mid \ \xmin \le y \le x \le \xmax } \to \reza $.
Moreover, the function $ f: \interval{\xmin}{\xmax} \to \reza $
is supposed to be approximately given,
and a function $ u: \interval{\xmin}{\xmax} \to \reza $
satisfying equation \refeq{volterra-inteq} needs to be determined.

In the sequel we suppose that the kernel function does not vanish on the
diagonal $ \xmin \le \varx = \vary \le \xmax $, and
\mywlog we may assume that
\begin{align*}
k\klasm{\varx,\varx} = 1 \for
\intervalarg{\varx}{\xmin}{\xmax}
%%\label{eq:k_eq_one}
\end{align*}
holds. 

Composite quadrature methods for the approximate solution of 
equation \refeq{volterra-inteq} are well-investigated if the \rhs $ f $ is exactly given,
see
\eg \mycitebtwo{Brunner}{van der Houwen}{Brunner_Houwen[86]},
\mycitea{Brunner}{04},
\mycitea{Lamm}{00},
\mycitea{Linz}{85}
or \myciteb{Hoog}{Weiss}{73}
and the references therein. 
A special class of composite quadrature methods for the approximate solution of
\refeq{volterra-inteq} is obtained by using 
in an appropriate manner \msms that usually are used to solve initial value problems for first order ordinary differential equations.
That class of methods is considered thoroughly in
Wolkenfelt (\mynocitea{Wolkenfelt}{79}, \mynocitea{Wolkenfelt}{81}). A related survey is given in
\mycitebtwo{Brunner}{van der Houwen}{Brunner_Houwen[86]},
and see \mycitec{Holyhead}{McKee}{Taylor}{75},
\myciteb{Holyhead}{McKee}{76}
and \mycitea{Taylor}{76} for related results.
In the present paper, the results and techniques presented in the two papers
by Wolkenfelt are modified and extended in order to analyze the regularizing properties 
of those \msms for Volterra integral equations \refeq{volterra-inteq} when perturbed \rhss are available only.
An a priori choice of the  \stepsize is considered, followed by the balancing principle as an adaptive choice of the \stepsize.
Finally, some numerical illustrations are presented.
\section{Numerical integration based on multistep methods}
In this section, as a preparation for the numerical solution of \voltinteqssh of the first kind \refeq{volterra-inteq} with smooth kernels, we introduce linear \msms for solving the associated direct problem. \Ftp we consider
equidistant \gridpoints 
\begin{align}
\xs = \xminpl s \mydeltax, \qquad s = 0, 1, \ldots,\N, 
\with \mydeltax = \tfrac{\diffxmaxxmin}{\N},
\label{eq:grid-points}
\end{align}
where $ \N $ denotes a positive integer.
In a first step we consider -- for each $ 1 \le \n \le \nmax $ -- the integral
\begin{align}
\myint{\myfun}{\xn}
:= \ints{\xmin}{\xn} 
{ \myfun\klami{\vary} }
{d \vary},
\label{eq:integration}
\end{align}
where $ \myfun: \interval{\xmin}{\xn} \to \reza $ is a given \cont function which may depend on $ n $.
In the course of this paper, this integral will be considered for the special case
$ \myfun(y) = k(\xn,y) u(y), \, \xmin \le x \le …\xn $; see Section \ref{msm-volterra-first-kind} for the details. 

The integral \refeq{integration} can be computed by solving the elementary ordinary differential equation
\begin{align}
\prim{\myphi}(\vary) = \myfun(\vary)
\for \intervalarg{\vary}{\xmin}{\xn}, \qquad \myphi(\xmin) = 0,
\label{eq:ivp-simple}
\end{align}
and then obviously 
$ \myint{\myfun}{\xn} = \myphi(\xn) $.  
Next we briefly introduce some basic facts about linear \msms to solve initial value problems for ordinary differential equations, with a notation that is adapted to the simple situation considered 
in \refeq{ivp-simple}.
For a thorough presentation of \msms (to solve initial value problems for 
ordinary differential equations in its general form), see \eg
\mynocitea{Plato}{03}, 
\mycitectwo{Hairer}{N{\o}rsett}{Wanner}{Hairer_etal[08]},
\mycitea{Henrici}{62},
or
\mycitea{Iserles}{08}.
\subsection{Introduction of \msms}
\label{msm_intro}
A linear \ksv, with an integer $ m \ge 1 $,
is determined by coefficients $ \mya_j \in \reza $ and $ \myb_j \in \reza $
for $ j =\myseqq{0}{1}{\emm} $, where $ \mya_\emm \neq 0 $ and
and $ \modul{\mya_0} + \modul{\myb_0} \neq 0 $.
When applied to problem \refeq{ivp-simple},
this scheme is of the form
\begin{align}
\mysum{j=0}{\emm} \mya_j  \myphi_{\myr + j} =
h \mysum{j=0}{\emm} \myb_j \myfun_{\myr + j}
\for  \myr \eq \myseqq{0}{1}{n-\emm},
\label{eq:multstep-def}
\end{align}
where $ n \ge \emm $, and $ \myfun_s = \myfun(\xs),  s = \myseqq{0}{1}{n} $ are given,
and the \stepsize $ h $ and the \gridpoints $ \xs $ are given by
\refeq{grid-points}. 
In addition we have $ \myphi_0 = 0 $, and the other starting values
$ \myphi_s \approx \myphi(\xs) $ for $ s = \myseqq{1}{2}{m-1} $ are determined by some procedure specified below
(see Example \ref{th:interpolatory-start-weights}).
The scheme \refeq{multstep-def} is used to compute approximations
$ \myphi_{r+m} \approx \myphi(\xn[r+m]) $ for $ r = \myseqq{0}{1}{n-m} $.
\begin{example}
\label{th:msm-examples}
\begin{myenumerate}
\item We first consider a well-known class of \msms of the form \refeq{multstep-def}, depending on three 
integers $ \myq, \, \mu $ and $ \mym $,
with $ 1 \le \myq \le m $ and $ 0 \le \mu \le m $. It
is obtained by integrating, for each $ 0 \le r \le n-m $,
the ordinary differential equation 
\refeq{ivp-simple} from $ x_{\myr+m-\myq} $ to $ x_{\myr+m} $.
For the integral of the resulting right-hand side, an interpolatory numerical integration scheme
with interpolation nodes $ \myseqq{x_r}{x_{r+1}}{x_{r+m-\mu}} $
is applied afterwards. This leads to
\begin{align}
\myphi_{\myr+m} - \myphi_{\myr+m-\myq}
=
\int_{x_{\myr+m-\myq}}^{x_{\myr+m}} \polr(y) \ dy,
\qquad
\myr = \myseqq{0}{1}{n-\emm},
\label{eq:quad-ansatz}
\end{align}
where $ \polr \in \Pi_{m-\mu} $ satisfies
$ \polr\klasm{\xs} = \myfun_{s} $ for
$ s = \myseqq{\myr}{\myr+1}{\myr + m-\mu} $.
This means that $ \myq h $ is the length of the interval used for the local integration, 
and $ m-\mu+1 $ is the number of \gridpoints used for the interpolation of the function 
$ \myfun $.
Prominent examples are obtained for 
$ \mu \in \{0,1\} $ and $ \myq  \in \{1,2\} $. Next some special cases are considered very briefly. For more details see, e.g., 
%%%\mynocitea{Plato}{03},
%%%\mycitectwo{Hairer}{N{\o}rsett}{Wanner}{Hairer_etal[08]},
%%\mycitea{Henrici}{62},
%%or \mycitea{Iserles}{08}.
the references given just before the present subsection.
%%%\label{msm_intro}

The Adams--Bashfort methods are obtained for $\myq=1, \, \mu = 1 $ and $ \mym \ge 1 $; for the special case $ m = 1 $ this in fact gives the composite forward rectangular rule.
The Adams--Moulton methods are obtained for $ \myq=1, \, \mu = 0 $ and $ \mym \ge 1 $, with the composite trapezoidal rule obtained for the special case $ \mym = 1 $.
The Nystr\"{o}m methods are given by $ \myq=2, \, \mu = 1 $ and $ \mym \ge 2 $. For $ \mym = 2 $ this gives the repeated midpoint rule.
Finally, the Milne--Simpson methods are obtained by $ \myq=2, \, \mu = 0 $ and $ \mym \ge 2 $, with the
repeated Simpson's rule obtained in the case $ \mym = 2 $.
Each of these methods is in fact of the form \refeq{multstep-def} and
leads to a repeated quadrature method
for solving \refeq{integration}, with
interpolation polynomials $\polr $ that, for $ m > \myq $,
have nodes outside the local integration 
interval $ \interval{\xn[r+\emm-\tau]}{\xn[r+\emm]} $.

\item Another class of linear \msms of the form \refeq{multstep-def}
are BDF methods (backward differentiation formulas), where the left-hand side in \refeq{ivp-simple} is replaced by a finite difference scheme. More precisely, for $ m $ fixed, approximations
$ \myphi_{\myr+m} \approx \myphi(\xn[r+m]) $ for $ r = \myseqq{0}{1}{n-m} $
are given by $ \myphi_{\myr+m} = \pol\klasm{x_{\myr+m}} $,
where $ \pol \in \Pi_{m} $ satisfies
$ \pol(\xs) = \myphi_{s} $
for $ s = \myseqq{\myr}{\myr+1}{\myr + m-1} $ and
$ \prim{\pol}\klasm{x_{\myr+m}} = \myfun_{\myr+m} $. For $ m = 1 $ this leads to the 
composite backward rectangular rule.
\remarkend
\end{myenumerate}
\end{example}
\subsection{Null stability, order of the method}
\label{order}
We next recall some basic notation for \msms applied to the simple initial value problem~ \refeq{ivp-simple}. 
\begin{myenumerate}
\item
The considered \msm is called nullstable, if the corresponding first \generating polynomial 
\begin{align}
\myrho(\xi) = \mya_m \xi^m + \mya_{m-1} \xi^{m-1} + \dots + \mya_0
\label{eq:multstep-gener-pol}
\end{align}
is a simple von Neumann polynomial, \ie
\begin{align}
\textup{(i)} \ \
\myrho(\xi) = 0 \ \textup{implies} \  \vert \xi \vert \le 1,
\qquad
\textup{(ii)} \ \
\myrho(\xi) =  0, \ \vert \xi \vert = 1 
\ \textup{implies} \
\prim{\myrho}(\xi) \neq  0.
\label{eq:nullstable}
\end{align}
This means that all roots of the characteristic polynomial $ \myrho $ belong to the closed unit disk, and each root on the unit circle is simple.
 
\item 
We next consider the local truncation error of the considered \msm. For technical reasons it is introduced here on arbitrary intervals $ \interval{c}{d} $
which in fact can be $ \interval{\xmin}{\xn} $ as above, or an interval of fixed length. 

For a continuous function $ \myfun: \interval{c}{d} \to \reza $ 
the local truncation error
is given by
\begin{align}
\locerr{\myfun,y,h} \defeq
\sum_{j=0}^m \mya_j 
\myphi(y+jh)
- h \sum_{j=0}^m \myb_j \myfun(y+jh),
\qquad
c \le y \le  d-mh, \ h > 0,
\label{eq:locerr}
\end{align}
where
$ \myphi: \interval{c}{d} \to \reza $ satisfies
$ \prim{\myphi}(\vary) = \myfun(\vary) $ for
$ \intervalarg{\vary}{c}{d} $ and $ \myphi(c) = 0 $.
The \msm~\refeq{ivp-simple} is by definition of (consistency) order $ \myqq $ with an integer $ \myqq \ge 1 $, if
on a fixed test interval $ \interval{c}{d} $
and for each 
sufficiently smooth function $ \myfun:[c,d] \to \reza $ and each $ c \le y < d $,
the estimate $ \locerr{\myfun,y,h} = \Landauno{h^{\myqq+1}} $ as $ h \to 0 $ holds.
A \msm is by definition of maximal order $ \myp \ge 1 $ if it is of order $ \myqq = \myp $ and not of order $ \myqq= \myp+1 $.
\end{myenumerate}
\begin{example}
\begin{myenumerate}
\item
\label{example-msm-order-repeated-quadrature-rules}
Each \msm of the special form \refeq{quad-ansatz} is clearly nullstable.
The \consistenceorder of this \msm is at least $ p = m-\mu + 1 $.  
The maximal order $ \myp $ may be larger in some cases.
For example, for $ \myq = 2, \, \mu = 0 $ and $ m = 2 $ (the Simpson's rule from  
the class of Milne--Simpson methods), the 
 \maxconsistenceorder is $ \myp = 4 $.
For those values of $  \myq $ and $ \mu $, the \ksvs coincide for $ m = 2 $ and $ m = 3 $ in fact.

\item The BDF methods are nullstable for $ 1 \le m \le 6 $, with respective maximal order $ \myp = m $.
\end{myenumerate}
\label{th:example-msm-order}
\end{example}
Next we consider the local truncation error \refeq{locerr}
on variable intervals $ \interval{c}{d} = \interval{\xmin}{\xn} $
and present uniform estimates.
As a preparation we introduce for $ \myqq \ge 0 $ and $ \myL \ge 0 $ the space
$ \Cpl{\myqq}[c,d] $ of functions $ u : \interval{c}{d} \to \reza $ that are $ \myqq $-times differentiable and in addition satisfy 
$ \modul{u^{(\myqq)}(y_1) -u^{(\myqq)}(y_2) } \le \myL \modul{y_1-y_2} $ for $ y_1, y_2 \in \interval{c}{d} $. Occasionally we also use the notation 
\begin{align*}
\UCpl{\myqq}[c,d] = \mycup_{\myL>0} \Cpl{\myqq}[c,d].
\end{align*}
\begin{lemma}
\label{th:locerr-estimate}
Consider a linear \msm \refeq{multstep-def} of \maxconsistenceorder $ \myp \ge 1 $
for solving the initial value problem \refeq{ivp-simple}.
Then for each Lipschitz constant $ \myL>0 $ and each 
$ 1 \le \myqq \le \myp $, the following estimate for the local truncation error holds:
\begin{align}
\locerr{\myfun,y,h} = \Landauno{h^{\myqq+1}}
\quad \text{as } h = \tfrac{\xmax-\xmin}{N} \to 0,
\label{eq:locerr-estimate}
\end{align}
uniformly for $ n, \myfun $ and $ y $ satisfying
$ m \le  n \le N, \myfun \in \Cpl{\myqq-1}[\xmin,\xn] $,
and $ \xmin \le y \le \xn -m h $.
\end{lemma}
\begin{proof}
A Taylor expansion of a function $ f \in \Cpl{\myqq}[c,d] $ on an interval $ \interval{c}{d} $ gives, for $ c \le y $ and $ y + h < d $, 
the following representation for the remainder:
$ R(f,p,y,h) \defeq f(y+h) - \mysumtxt{s=0}{p} \tfrac{\ableit{f}{s}(y)}{s!} h^k
= \tfrac{\ableit{f}{p}(\xi) -\ableit{f}{p}(y)}{p!} h^p $.
This means 
$ R(f,p,y,h) = \Landauno{h^{p+1}} $ as $ h > 0, h \to 0 $,
uniformly for arbitrary finite intervals $ \interval{c}{d} $, for $ f \in \Cpl{\myqq}[c,d] $, and $ c \le y < d $.
After these preparations we now consider the special situation in the lemma.
From appropriate Taylor expansions for 
$ \myfun $ and $ \myphi $, and making use of the 
consistency equations corresponding to the \msm
\refeq{multstep-def} for solving \refeq{ivp-simple},
we finally arrive at
$ \locerr{\myfun,y,h} =
\mysumtxt{j=0}{\emm} \mya_j  R_j -
h \mysumtxt{j=0}{\emm} \myb_j \Rtil_j
= \Landauno{h^{\myqq+1}}
$,
uniformly for $ n, \myfun $ and $ y $ as given in the statement of the lemma,
where
$ R_j = R(\myphi,p,y,jh) $ and
$ \Rtil_j = R(\myfun,p-1,y,jh) $.
\end{proof}

\bigskip \noindent
We note that the considered intervals $ [\xmin,\xn] $ in Lemma \ref{th:locerr-estimate}
depend on $ h $, and we do not require $ \xn $ to be fixed. This causes no problem
in \refeq{locerr-estimate}, however, since the estimates of the local truncation error are considered uniformly there.

The basic convergence result in \msm theory is as follows: each nullstable linear \msm \refeq{multstep-def} of \consistenceorder $ \myqq \ge 1 $ is convergent of order $ \myqq $. Details are given in Section \ref{global_error_expansion}.

\subsection{Reflected coefficients / polynomials}
As a preparation we introduce some more notation.
We assume that at least one of the coefficients on the right-hand side of 
\refeq{multstep-def} does not vanish, and we identify the leading nonvanishing coefficient
then:
let $ 0 \le \mymu \le m $ such that
\begin{align} 
b_{m-\mu+1} = \cdots = b_{m-1} =  b_m = 0, \quad b_{m-\mu} \neq 0.
\label{eq:sigma_zero_roots}
\end{align}
In the sequel we make use of a relation between linear difference equations and discrete convolution equations. As a preparation we
consider infinite sequences of reflected coefficients
$ (\alpha_j)_{j \ge 0} $
and $ (\beta_j)_{j \ge 0} $
of the \msm under consideration:
\begin{align}
\alpha_j = \left\{
\begin{array}{rl}
\mya_{m-j}, & j \le m, \\
0, & j > m,
\end{array}
\right.
\qquad
\beta_j = \left\{
\begin{array}{rl}
\myb_{m-\mu-j}, & j \le m-\mu, \\
0, & j > m+\mu.
\end{array}
\right.
\label{eq:alp-bet-def}
\end{align}
In addition we introduce sequences
$ \alpinv_0, \alpinv_1, \dots $
and 
$ \mygamma_0, \mygamma_1, \ldots $
by the following discrete convolution equations:
\begin{align}
\sum_{s=0}^r \alpha_{r-s} \alpinv_s = \delta_{0r},
\qquad
\sum_{s=0}^r \alpha_{r-s} \mygamma_s = \beta_{r},
\for r = 0, 1, \ldots,
\label{eq:gamma-def}
\end{align} 
where $ \delta_{0r} $ denotes the Kronecker symbol, \ie
we have $ \delta_{00} = 1 $ and
$ \delta_{0r} = 0 $ for each $ r \neq 0 $.
There is a relation between those discrete convolutions 
and the products of the associated (formal) power series: for
\begin{align}
\alpha(\xi) = \mysum{n=0}{m} \alpha_n \myxi^n, \quad 
\beta(\xi) = \mysum{n=0}{m-\mu} \beta_n \myxi^n, \quad 
\gamma(\xi) = \mysum{n=0}{\infty} \gamma_n \myxi^n,
\label{eq:fps-abc}
\end{align}
we have
\begin{align*}
\mfrac{1}{\alpha\klasm{\myxi}} = \mysum{n=0}{\infty} \alpinv_n \myxi^n,
\quad
\alpha(\xi) \gamma(\xi) = \beta(\xi).
\end{align*}
In addition, there is a relation between products of (formal) power series considered in \refeq{fps-abc} on one side and the products of associated semicirculant matrices on the other side. This relation will be tacitly used in the sequel. For an introduction to that topic, see, \eg 
\mycitea{Henrici}{74}.

It follows from \refeq{gamma-def} and standard results for difference equations
(see, \eg Lemma 5.5 on p.~242 in \mycitea{Henrici}{62})
that a nullstable \msm satisfies
\begin{align} 
\alpinv_n = \Landauno{1},
\quad
\mygamma_n = \Landauno{1} \as n \to \infty.
\label{eq:gamman-bounded}
\end{align} 
In the stability analysis to be considered, the coefficients of the inverse \fps
$ 1/\beta\kla{\myxi} $ and $ 1/\gamma\kla{\myxi} $ also play a significant role.
Their behavior will be considered later. 
\subsection{A global error representation}
\label{global_error_expansion}
We next present a global error representation in terms of linear combinations of local truncation errors, as well as the starting errors. This representation will be crucial in the subsequent analysis.
\begin{lemma}
Consider a nullstable linear \msm \refeq{multstep-def} 
for solving the initial value problem \refeq{ivp-simple}.
Then we have the error representation
\begin{align}
\myphi_n = \myint{\myfun}{\xn} - 
\sum_{s=0}^{n-m} \alpinv_{n-m-s} \locerr{\myfun,\xs,h} + R, \qquad 
\modul{R} \le C \max_{0 \le r \le m-1} \modul{ \myphi_\myr - \myphi(x_\myr)}, 
\label{eq:ode_global_error_expansion}
\end{align}
where $ n \ge m $, and
$ \locerr{\myfun,\xs,h} $ denotes the local truncation error at the \gridpoint $ \xs $,
\cf \refeq{locerr}.
The constant $ C $ in \refeq{ode_global_error_expansion}
depends on $ m $ and the 
bounds for $ (\alpha_n)_{n\ge 0} $ and $ (\alpinv_n)_{n\ge 0} $ only.
\label{th:ode_global_error_expansion}
\end{lemma}
\begin{proof}
Let $ e_r = \myphi_\myr - \myphi(x_\myr) $ for $ r = \myseqq{0}{1}{n} $. 
We have $ \mysumtxt{j=0}{m} a_j e_{r+j} = g_r $ for $ r = \myseqq{0}{1}{n-m} $, 
where $ g_r \defeq -\locerr{\myfun,\xr,h} $.
A reformulation gives 
$ \mysumtxt{i=0}{r} \alpha_{r-i} e_{i+m} = g_r - \mysumtxt{i=r-m}{-1} \alpha_{r-i} e_{i+m}
$
for $  r = \myseqq{0}{1}{n-m} $,
which in matrix formulation reads as follows:
\begin{align*}
\left(
\begin{array}{c@{\quadsm}c@{\quadsm}c@{\quadsm}c@{\quadsm}c@{\quadsm}c@{\quadsm}c}
\alpha_0 & 0 & \cdots & \cdots & \cdots & \cdots & 0 \\
\alpha_1 & \alpha_0 & \ddots & & & & \vdots \\ 
\vdots & \ddots & \ddots & \ddots &  && \vdots \\
\alpha_m & & \ddots & \ddots & \ddots & & \vdots \\
0 & \ddots & & \ddots & \ddots & \ddots &\vdots \\
\vdots & \ddots & \ddots & & \ddots & \ddots & 0 \\
0 & \cdots & 0 &  \alpha_m & \cdots & \alpha_1 & \alpha_0 \\
\end{array} \right) \cdott 
\left(\begin{array}{c}
e_m \\ e_{m+1} \\
\vdots \\
\vdots \\
\vdots \\
\vdots \\
e_n 
\end{array} \right)
=
\left(\begin{array}{c}
g_0 + \Landaubi{\max_{0 \le r \le m-1} \modul{e_r}} \\
\vdots \\
g_{m-1} + \Landaubi{\max_{0 \le r \le m-1} \modul{e_r}} \\
g_m \\
\vdots \\
\vdots \\
g_{n-m} 
\end{array} \right).
\end{align*}
The desired result now follows from the fact that the inverse of the semicirculant system matrix is given by
\begin{align*}
\left(
\begin{array}{c@{\hspace{2mm}}c@{\hspace{2mm}}c@{\hspace{2mm}}c}
\alpinv_0 & 0 & \cdots & 0 \\
\alpinv_1 & \alpinv_0 & \ddots & \vdots \\ 
\vdots & \ddots & \ddots &  0 \\
\alpinv_{n-m} & \cdots & \alpinv_1  & \alpinv_0  \\
\end{array} \right).
\end{align*}
This completes the proof.
\end{proof}
\begin{remark}
It immediately follows from Lemmas \ref{th:locerr-estimate}
and \ref{th:ode_global_error_expansion}
as well from \refeq{gamman-bounded} that, 
under the conditions stated in those lemmas, we have
$ \myphi_n = \myint{\myfun}{\xn} + \Landauno{h^{\myqq}} $ for $ n \ge m $,
provided that the starting errors are $ \Landauno{h^{\myqq}} $.
This result, however, does not allow optimal error estimates for the approximate inversion of Volterra integral equations of the first kind to be considered in this paper, so we make use of
\refeq{ode_global_error_expansion} instead.
We note that in the papers by Wolkenfelt (\mynocitea{Wolkenfelt}{79}, \mynocitea{Wolkenfelt}{81}), a global error expansion with an integral representation is used for the inversion process to obtain best possible error estimates.
The latter approach, however, requires 
stronger smoothness assumptions on the involved functions
than our 
approach based on \refeq{ode_global_error_expansion} does.
\remarkend
\end{remark}
\subsection{Explicit representation of the values $ \myphi_r $}
For the numerical analysis to be considered later on we need to express the
values $ \myseqq{\myphi_\mym}{\myphi_{\mym+1}}{\myphi_n} $
generated by the \msm
\refeq{multstep-def}
in terms of the numbers $ \myfuns $
and the starting values
$ \myseqq{\myphi_1}{\myphi_2}{\myphi_{\mym-1}} $
(as indicated, we always choose $ \myphi_0 = 0 $).
To simplify notation somewhat and to adapt our notation to the existing literature on the topic, we shall assume that the starting values are of the form
\begin{align}  
\myphi_r = h \mysum{s=0}{\myd-1} \wstart{rs} \myfuns,
\quad
\myr = \myseqq{1}{2}{m-1},
\label{eq:start_values_rep}
\end{align} 
where
$ \wstart{rs} \in \reza $ for $ \myr = \myseqq{1}{2}{\myd-1} $ and
$ s = \myseqq{0}{1}{\myd-1} $, are starting weights which are independent of $ h $ and which will be specified below.
We note that in \refeq{start_values_rep},
each starting value $ \myphi_{r} \  (1 \le r \le m-2 $) obviously may depend on future states, in general, which is rather unnatural for a Volterra type problem. 
Such an approach, however, allows sufficiently good accuracy of those starting values.

As a further preparation for Lemma \ref{th:phi_proofrep} considered below, we introduce weights needed in that lemma:
\begin{myenumerate}
\item Consider
\begin{align} 
\myomega_{ns} = \mygamma_{n-\mu-s} 
\quad  \text{for } m \le s \le n-\mu,
\quad  m + \mu \le n < \infty,
\label{eq:omega-nonstart-def}
\end{align} 
where the numbers $ \mygamma_{s} $ are given by \refeq{gamma-def}.

\item
For $ n \ge m $ we next consider starting weights $ \myomega_{ns}, 0 \le s \le m-1 $. For $ s $ fixed, they are
recursively determined by the following inhomogeneous discrete convolution equation, 
\begin{align}
\sum_{t=0}^n \alpha_{n-t} \myomega_{ts} 
= \left\{
\begin{array}{rl}
\beta_{n-\mu-s}, & n \ge \mu + s, \\
0, & n < \mu + s, \\
\end{array}
\right.
\qquad n = m, m+1, \ldots \ . 
\label{eq:omega-start-def}
\end{align} 
The weights introduced in \refeq{omega-nonstart-def}, \refeq{omega-start-def} are uniformly bounded in case of a nullstable method,
\ie
\begin{align} 
\sup_{m + \mu \le n < \infty \atop 0 \le s \le n-\mu} \modul{\myomega_{ns}} < \infty.
\label{eq:omega-bounded}
\end{align} 
This follows, similarly to \refeq{gamman-bounded},
\refeq{omega-nonstart-def},
from standard results for difference equations.
\end{myenumerate}
We are now in a position to represent the \msm \refeq{multstep-def} in quadrature form. Note that the 
numbers $ \myseqq{\myfun_0}{\myfun_1}{\myfun_{n-\mu}} $ considered in the following lemma do not necessarily coincide with the values of the previously considered function $ \myfun: \interval{\xmin}{\xn} \to \reza $ at the given \gridpoints.
\begin{lemma} 
Let 
 $ \myseqq{\myphi_1}{\myphi_2}{\myphi_{n}} $
and $ \myseqq{\myfun_0}{\myfun_1}{\myfun_{n-\mu}} $
 be arbitrary two sequences 
of real numbers satisfying
\refeq{start_values_rep}
and the \msm recurrence \refeq{multstep-def} 
with $ n \ge \emm+\mu $ and $ \myphi_0 = 0 $.
Then the following identity holds:
\begin{align}  
\myphi_{n} =
h \mysum{s=0}{n-\mu} \myomega_{ns} \myfuns,
\label{eq:phi_proofrep}
\end{align} 
where the weights $ \myomega_{ns} $ 
are given by \refeq{omega-nonstart-def} and \refeq{omega-start-def}.
\label{th:phi_proofrep}
\end{lemma}
\begin{proof} It follows by induction that
a representation of the form \refeq{phi_proofrep} with some 
weights $ \myomega_{ns} $
exists in general.
The special representations of the weights given in
\refeq{omega-nonstart-def} and \refeq{omega-start-def}
are then obtained by considering
in \refeq{phi_proofrep}
the standard basis of $ \reza^{n-\mu+1} $ to represent
$ \myseqq{\myfun_0}{\myfun_1}{\myfun_{n-\mu}} $.
Details are omitted.
\end{proof}

\bigskip \noindent
A quadrature method \refeq{phi_proofrep} generated by 
a \msm \refeq{multstep-def}
with starting values as in
\refeq{start_values_rep}
is called $ (\varrho,\sigma) $-reducible;
see, e.g., 
\mycitebtwo{Brunner}{van der Houwen}{Brunner_Houwen[86]},
\mycitea{Taylor}{76} or
Wolkenfelt (\mynocitea{Wolkenfelt}{79}, \mynocitea{Wolkenfelt}{81}).
\subsection{A starting quadrature procedure}
For \msms \refeq{multstep-def} to solve the initial value problem \refeq{ivp-simple},
we next consider, for $ \emm \ge 2 $, the determination of starting values $ \myseqq{\myphi_1}{\myphi_2}{\myphi_{m-1}} $ of the form \refeq{start_values_rep} that 
have the approximation properties required in
Lemma~\ref{th:ode_global_error_expansion}.
A standard procedure
is presented in the following example.
\begin{example}
\label{th:interpolatory-start-weights}
We consider \refeq{start_values_rep} for fixed $ \myr \in \inset{\myseqq{1}{2}{\mym}} $ with $ \mym \ge 1 $. The case $ r = m $ is not considered there in fact but this will be needed for the computation of 
initial approximations to the solution of the
Volterra integral equation of the first kind \refeq{volterra-inteq}
considered below. Note also that in the case $ r = m $ there is a notational conflict
with \refeq{omega-bounded}, for $ n = m $ there. We will take care of this in every application.

We now consider
an interpolatory quadrature method for the integral
$ \int_{\xmin}^{\xr} \myfun(y) dy $ using interpolation nodes
$ \myseqq{\xn[0]}{\xn[1]}{\xn[\myd-1]} $.
This in fact means that
the resulting quadrature scheme 
$ \myphi_r = \linebreak  h \mysumtxt{s=0}{\myd-1} \wstart{rs} \myfun(\xs)
\approx \int_{\xmin}^{\xr} \myfun(y) dy $
is exact for all polynomials $ \myfun $ of degree $ \le \myd-1 $, with quadrature weights
that are given by the following nonsingular linear system of equations:
\begin{align}
\left( \begin{array}{c@{\quad}c@{\ \ }c@{\ \ }c@{\ \ }c}
1 & 1 & 1 & \cdots & 1 \\
0 & 1 & 2 & \cdots & \myd-1 \\
0 & 1 & 4 & \cdots & \klasm{\myd-1}^2 \\
0 & 1 & 9 & \cdots & \klasm{\myd-1}^3 \\
\vdots & \vdots & \vdots & & \vdots \\
0 & 1 & 2^{\myd-1} & \cdots & \klasm{\myd-1}^{\myd-1}
\end{array} \right)
\left( \begin{array}{@{}c@{}}
\wstart{r0} \\
\wstart{r1} \\
\wstart{r2} \\
\vdots \\
\wstart{r,\myd-1}
\end{array}
\right)
=
\left( \begin{array}{@{}c@{}}
r \\
\lfrac{r^2}{2} \\
\lfrac{r^3}{3} \\
\vdots \\
\lfrac{r^\myd}{\myd}
\end{array}
\right).
\label{eq:start-weights-system-matrix}
\end{align}
We next study the error of this quadrature scheme, and for this purpose  
let $ 1 \le \myqq \le m $ and $ L > 0 $ be fixed.
For functions $ \myfun \in \Cpl{\myqq-1}[\xmin,\xm] $
and $ \pol \in \Pi_\emm $ with
$ \pol(\xs) = \myfun(\xs) $
for $ s = \myseqq{0}{1}{\mym-1} $,
we have
\begin{align}
\max \inset{\modul{\pol(y) - \myfun(y)} \mid \xmin \le y \le \xm}
= \Landauno{h^{\myqq}}
\label{eq:polapp-start}
\end{align}
uniformly \wrt the considered class of functions $ \myfun $. This follows from elementary interpolation theory:
for each $ y \not \in \inset{\myseqq{\xn[0]}{\xn[1]}{\xn[m-1]}} $
we have
$ \pol(y) - \myfun(y) = \myfun[\xn[0],\xn[1],\ldots,\xn[m-1],y] w(y) $, 
where 
the first factor on the right-hand side denotes a divided difference, and 
$ w(y) = (y-\xn[0]) \cdots (y-\xn[m-1]) $.
It follows by induction that
$ \modul{\myfun[\xn[0],\ldots,\xn[m-1],y]}
\le \kappa / h^{m-\myqq} $ holds, with the constant
$ \kappa = 2^{m-\myqq} L/ (m-1)! $.
This finally gives \refeq{polapp-start}. 

From \refeq{polapp-start}
we immediately obtain
$ \myphi_\myr -
\int_\xmin^{\xr} \myfun(y) dy
= \Landauno{h^{\myqq+1}} $
for $ \myr = \myseqq{1}{2}{\mym} $
uniformly \wrt the considered class of functions $ \myfun $.
Note that the assumption $ \myqq \le \mym $ made here is no serious restriction;
see Remark \ref{th:strong-root-cond-remarks} below for details.
Note also that the starting weights $ \wstart{rs} $
given by \refeq{start-weights-system-matrix}
do not depend on $ h $ and $ n $.
\end{example}
We summarize the results of
Lemma \ref{th:ode_global_error_expansion}, 
Lemma \ref{th:phi_proofrep} and Example \ref{th:interpolatory-start-weights}.
\begin{corollary}
Consider a nullstable linear \msm \refeq{multstep-def}, with $ n \ge \emm $.
Let the weights $ \omega_{ns} $ for $ n \ge m $ be given by
\refeq{omega-nonstart-def} and \refeq{omega-start-def},
with starting weights 
$ \omega_{ns} $ for $ n \le m-1 $ be given by
as in Example~\ref{th:interpolatory-start-weights}.
Then for each
$ 1 \le \myqq \le m $ and each Lipschitz constant $ L > 0 $ we have 
\begin{align*}
h \mysumtxt{s=0}{n-\mu} \myomega_{ns} \myfun(\xs)
= \myint{\myfun}{\xn} -
\sum_{s=0}^{n-m} \alpinv_{n-m-s} \locerr{\myfun,\xs,h}
+ \Landauno{h^{\myqq+1}}
\end{align*}
uniformly for $ n $ and $ \myfun $
satisfying $ m+\mu \le  n \le N $ and $ \myfun \in \Cpl{\myqq-1}[\xmin,\xn] $. 
\label{th:ode_msm_representation}
\end{corollary}
\section{Linear \msms for perturbed first kind Volterra integral equations}
\label{msm-volterra-first-kind}
\subsection{Some preparations}
We now return to the first kind Volterra integral equation \refeq{volterra-inteq}. For the numerical approximation we consider this equation at equidistant nodes
$ \xn = \xminpl \n \mydeltax, \, n = 1, 2, \ldots,\N $ with 
$ \mydeltax = \tfrac{\diffxmaxxmin}{\N} $, \cf\refeq{grid-points}. 
For each $ n = m, m+1, \ldots,\N $,
the resulting integral
$ \myint{\myfun}{\xn} = 
\intstxt{\xmin}{\xn}{\myfun(y)}{dy} $
with $ \myfun(y) =  k\klasm{\xn,y} u\klasm{y} $ for $ \intervalargno{y}{\xmin}{\xn} $
is approximated by the \msm
\refeq{multstep-def} 
under consideration.

In the sequel we suppose that the \rhs of equation \refeq{volterra-inteq} is only approximately given, with
\begin{align}
\modul{ \fndelta - f\klasm{\xn} } \le \delta
\for n = \myseqq{1}{2}{\nmax},
\label{eq:rhs-assump}
\end{align}
where $ \delta > 0 $ is a known noise level.

For the main convergence results 
we impose the following conditions.
\myassump{
For the Volterra integral equation \refeq{volterra-inteq} of the first kind and a given
\ksv with $ m \ge 1 $ (see \refeq{multstep-def}), we introduce the following
assumptions and notations.
\begin{myenumerate_indent}
\item 
\label{item:order-p-m}
The considered \ksv 
with $ m \ge 1 $ is nullstable and has \maxconsistenceorder $ 1 \le \myp \le m $.

\item
\label{item:sigma-strong-root-condition}
The second \generating polynomial
\begin{align}
\sigma\klasm{\xi} \defeq b_{m-\mu} \xi^{m-\mu} + b_{m-\mu-1} \xi^{m-\mu-1} + \dots + b_0,
\label{eq:msm-second-gener-pol}
\end{align}
with $ \mu $ as in \refeq{sigma_zero_roots}, 
is a \schurpolynomial:
\begin{align}
\sigma(\xi) = 0 \Longrightarrow \vert \xi \vert < 1 \qquad (\xi \in \koza),
\label{eq:sigma-strong-root-condition}
\end{align}
\ie all roots of the polynomial $ \sigma $ belong to the open unit disk.

\item
\label{item:assump-u}
There exists a solution $ u: \interval{\xmin}{\xmax} \to \reza $ to the integral equation
\refeq{volterra-inteq}, with $ u \in \UCpl{\myqq-1}[\xmin,\xmax] $ 
for some $ 1 \le \myqq \le \myp $
(for the definition of the considered function space, see Section \ref{order}).

\item
\label{item:assump-k-smooth}
For some integer $ \Nmin \ge m $ and $ \hmax = \tfrac{\diffxmaxxmin}\Nmin $,
the kernel function satisfies
$ k \in C^{\myqq}(\eset) $, where
$ \eset = \inset{(x,y) \mid \xmin \le y \le x \le \xmax \ \text{or} \ 
\xmin \le x,y \le \xminpl m \hmax} $.

\item
\label{item:assump-k=1}
There holds $ k\klasm{\varx,\varx} = 1 $ for each $ \intervalarg{\varx}{\xmin}{\xmax}
$. 

\item
\label{item:grid-points}
For a given \stepsize $ h = \tfrac{\diffxmaxxmin}{\nmax} $ with some integer $ \nmax \ge \Nmin $, let
$ \myseqq{\xn[0]}{\xn[1]}{\xn[\nmax]} $ be uniformly distributed \gridpoints 
given by \refeq{grid-points}.

\item
The values of the \rhs of equation \refeq{volterra-inteq}
are approximately given by \refeq{rhs-assump}.

\end{myenumerate_indent}
}{msm-assump}

Next we give some comments on the \schurpolynomial property considered in 
item \ref{item:sigma-strong-root-condition} of Assumption~\ref{th:msm-assump}.
\begin{remark}
\label{th:strong-root-cond-remarks}
\begin{myenumerate}
\item
In the stability analysis to be considered, the coefficients of the inverse \fps
\begin{align}
\mfrac{1}{\beta\klasm{\myxi}}
=
\mysum{n=0}{\infty} \betinv_n \myxi^n,
\qquad
\mfrac{1}{\gamma\klasm{\myxi}}
=
\mfrac{\alpha\klasm{\myxi}}{\beta\klasm{\myxi}}
=
\mysum{n=0}{\infty} \gaminv_n \myxi^n,
\label{eq:betinv-gaminv-def}
\end{align}
of the \genfuncs
$ \beta\kla{\myxi} $ and
$ \gamma\kla{\myxi} $, respectively
(see \refeq{fps-abc}),
play a significant role.
The \schurpolynomial condition \refeq{sigma-strong-root-condition} implies that
$ \lfrac{1}{\beta\klasm{\myxi}} $
is analytic in an open set of the complex plane that contains a disk $ \inset{\xi \in \koza \mid \modul{\xi} \le R } $ for
some $ R > 1 $, and Cauchy's integral theorem then implies that
the coefficients $ \betinv_n $ in \refeq{betinv-gaminv-def}
decay exponentially, \ie
\begin{align}
\betinv_n = \Landauno{\tau^{n}} \as n \to \infty
\ \textup{for some  } 0 < \tau < 1,
\label{eq:betainv-decay}
\end{align}
with $ \tau = 1/R $ in fact. From this 
$ \mygammainvnmu = \Landauno{\tau^{n}} $ as $ n \to \infty
$ follows easily.

\item
It is easy to see that for the $ \emm$-step Adams--Bashfort methods with $ 1 \le m \le 3 $, and the $ \emm $-step Nystr\"{o}m method with $ 2 \le m \le 3 $ as well,
the second characteristic polynomial $ \sigma $ is a \schurpolynomial (see condition \refeq{sigma-strong-root-condition}), respectively.
In addition, \refeq{sigma-strong-root-condition} is obviously satisfied by the BDF methods.

\item
The \schurpolynomial condition \refeq{sigma-strong-root-condition} is violated for each \msm of class \refeq{quad-ansatz}
with $ \mu = 0 $ (the implicit case) and with
\maxconsistenceorder $ \myp > m $.
%%%
%%
More generally, it is an essential observation made by \myciteb{Gladwin}{Jeltsch}{74} 
that the second \genpol $ \sigma $ is even not a simple von Neumann polynomial
in that situation, with the case $ \mym = \myq = 1 $ (the repeated trapezoidal rule) as an exception. In addition,
the associated scheme for solving Volterra integral equations of the first kind introduced below is necessarily divergent then, in general.
For the mentioned exception $ \mym = \myq = 1 $, the associated second \genpol is obviously a simple von Neumann polynomial but not a \schurpolynomial.

As a consequence of the observations made in the beginning of part (c) of this remark, in
the special situation $ \mu = 0 $ in the local quadrature approach \refeq{quad-ansatz},
it is no loss of generality to restrict the considerations to \ksvs of \maxconsistenceorder $ 1\le \myp \le m $
(see item \ref{item:order-p-m} of Assumption \ref{th:msm-assump}). 

\item
We note that in Wolkenfelt (\mynocitea{Wolkenfelt}{79}, \mynocitea{Wolkenfelt}{81}), the second characteristic polynomial $ \sigma $ is required to be a von Neumann polynomial, not a Schur polynomial which is the assumption made in the present paper (see condition \refeq{sigma-strong-root-condition}). The latter assumption results in noise amplification terms which in general are smaller than for 
 Neumann polynomials $ \sigma $. Those terms in fact are, up to some factor, of the form $ \lfrac{\delta}{h} $. In addition, the Schur polynomial assumption on $ \sigma $ allows to use a proof technique which in part is much simpler than the elaborated technique used in \mynocitea{Wolkenfelt}{79}.
\remarkend
\end{myenumerate}
\end{remark}

\subsection{The numerical scheme}
We now consider, under the conditions given in Assumption \ref{th:msm-assump}, the following scheme for the numerical solution of
a Volterra integral equation \refeq{volterra-inteq}:
\begin{algorithm}
\label{th:msm-volterra-scheme}
\begin{myenumerate}
\item Determine $ \emm $ \inapps $ \undelta[s] \approx u(\xs) $ for $ s = 0, 1, \ldots, \emm-1 $ by solving the following linear system of $m$ equations,
\begin{align}
h \mysum{s=0}{m-1} \wstart{ns} k\kla{\xn,\xs} \undelta[s]
=
\fndelta, \qquad n = 1, 2, \ldots, m,
\label{eq:msm-noise-start-values}
\end{align}
where the starting weights $ \wstart{ns} $ are given by \refeq{start-weights-system-matrix}, with $ r $ replaced by $ n $ there.

\item Determine then recursively, with $ \mu $ given by \refeq{sigma_zero_roots},
approximations $ \undelta[n-\mu] \approx u(\xn[n-\mu]) $ for 
$ n = m+\mu, \ldots, \nmax $ with $ \nmax \ge \Nmin $ by the following scheme.

For $ n $ fixed and 
$ \myseqq{\undelta[m]}{\undelta[m+1]}{\undelta[n-\mu-1]} $ already being computed, the following steps have to be employed to determine $ \undelta[n-\mu] $:
\begin{mylist_indent}
\item Set $ \psidel[s] = k(\xn,\xs) \undelta[s] $ for $ s = 0, 1, \ldots, n-\mu-1 $,
\item set $ \phidel[0] = 0 $, and compute (for $ \emm \ge 2$) $ \phidel = h \mysumtxt{s=0}{\mym-1} \wstart{rs} \psidel $
for $ \myr = \myseqq{1}{2}{m-1} $, \cf \refeq{start_values_rep},
where the starting weights $ \wstart{rs} $ are given by \refeq{start-weights-system-matrix},

\item compute recursively $ \phidel[\myr + \emm] $
for $ \myr = \myseqq{0}{1}{n-\emm-1} $ 
by using
on the interval $ [\xmin,\xn ] $ the perturbed version of the \mss \refeq{multstep-def}:
\begin{align}
\mysum{j=0}{\emm} \mya_j  \phidel[\myr + j] =
h \mysum{j=0}{\emm-\mu} \myb_j \psidel[\myr + j]
\for  \myr \eq \myseqq{0}{1}{n-\emm-1},
\label{eq:multstep-def-noise}
\end{align}

\item set $ \phidel[n] = \fndelta $, 

\item compute $ \psidel[n-\mu] $ by using the 
identity
\refeq{multstep-def-noise}
for $ r = n-\emm $,
\item compute $ \undelta[n-\mu] = \psidel[n-\mu]/k(\xn,\xn[n-\mu]) $.
\remarkend
\end{mylist_indent}
\end{myenumerate}
\end{algorithm}
\begin{remark}
\label{th:msm-volterra-scheme-remark}
\begin{myenumerate}
\item
Note that due to \ref{item:assump-k=1} in Assumption \ref{th:msm-assump}, 
for $ h $ sufficiently small we have $ \modul{k\kla{\xn,\xn[n-\mu]}} \allowbreak \ge C > 0 $ for each $ n $, independently of $ h $.
Thus the numerical procedure considered above can 
in fact be used for the stable computation of $ \undelta[n-\mu] $.

\item
The scheme \refeq{msm-noise-start-values} results from 
the quadrature method considered in Example \ref{th:interpolatory-start-weights},
applied to the integral
$ \int_\xmin^{\xn} k(\xn,y)u(y) dy $ 
for $ n = \myseqq{1}{2}{m} $.

\item
It immediately follows from Lemma \ref{th:phi_proofrep} that the approximations obtained by Algorithm \ref{th:msm-volterra-scheme} satisfy
\begin{align}
h \mysum{s=0}{n-\mu} \myomega_{ns} k\kla{\xn,\xs} \cdott \undelta[s]
=
\fndelta, \qquad n = \myseqqsh{m+\mu}{m+\mu+1}{\nmax},
\label{eq:msm-noise-proofrep}
\end{align}
where the weights $ \omega_{ns} $
are given by \refeq{omega-nonstart-def} and \refeq{omega-start-def},
 respectively.
The representation \refeq{msm-noise-proofrep} will be used in the proof of the main result, 
\cf Theorem \ref{th:main-msm}. In addition, for \msms of the form \refeq{quad-ansatz}, those weights can also be easily computed in practice, and 
\refeq{msm-noise-proofrep} can then be used for the practical implementation of 
\refeq{multstep-def-noise}. For an illustration see Example \ref{th:num_example_3} below.

\item The considered numerical scheme in Algorithm \ref{th:msm-volterra-scheme}
is quite universal and can be simplified in special cases. 
For example, for the backward rectangular rule
(which is the 1-step BDF method) considered in part (b) of Example \ref{th:msm-examples}, 
an implementation of Algorithm \ref{th:msm-volterra-scheme} without the 
starting procedure considered in (a) there is possible.
This means, however, that no approximation $ \undelta[0] $ will be available then.
\remarkend
\end{myenumerate}
\end{remark}
\subsection{Uniqueness, existence and approximation properties
of the \inapps}
\label{integsolvec}
We now consider uniqueness, existence as well as the approximation
properties of the \inapps $ \myseqq{\undelta[0]}{\undelta[1]}{\undelta[\mym-1]} $.
In a first step we consider in more detail the corresponding linear system of equations
\refeq{msm-noise-start-values}. This system of equations can be written in the form
\begin{align}
& h
\myoverbrace{
\left( \begin{array}{@{\ }c@{\ \  }c@{\ \ }c@{\ \ }c@{\ }}
\wstart{10} k\klasm{x_1,x_0} &
\wstart{11} k\klasm{x_1,x_1} &
\cdots &
\wstart{1,\mym-1} k\klasm{x_1,x_{\mym-1}}
\\ [7mm]
\wstart{20} k\klasm{x_2,x_0} &
\wstart{21} k\klasm{x_2,x_1} &
\cdots &
\wstart{2,\mym-1} k\klasm{x_2,x_{\mym-1}}
\\ [7mm]
\vdots & \vdots & & \vdots
\\ [7mm]
\wstart{\mym0}  k\klasm{x_\mym,x_0} &
\wstart{\mym1} k\klasm{x_\mym,x_1} &
\cdots &
\wstart{\mym,\mym-1} k\klasm{x_\mym,x_{\mym-1}}
\end{array} \right)}{\dis =: \ \Sh}
\left( \begin{array}{@{\ }c@{\ }}
\undelta[0] \\[1mm] \undelta[1] \\[1mm] \vdots \\[1mm]  \undelta[\mym-1]\end{array} \right)
=
\left( \begin{array}{@{\ }c@{\ }}
f_1^\delta \\[1mm] f_2^\delta \\[1mm] \vdots \\[1mm] f_{\mym}^\delta
\end{array} \right).
\label{eq:start-values-system-matrix}
\end{align}
Note that the matrix $ \Sh \in \myrnn[\mym] $
introduced in \refeq{start-values-system-matrix}
depends on the stepsize $ h $.
\begin{proposition}
The system matrix $ \Sh $ in \refeq{start-values-system-matrix}
is regular for sufficiently small values of $ h $, and
$ \maxnorm{\Shinv} = \Landausm{1} $ as $ h \to 0 $.
\label{th:start-values-system-matrix}
\end{proposition}
\proof
We first consider the situation $ k \equiv 1 $.
In a first step we observe that 
\refeq{start-weights-system-matrix} applied for $ r = \myseqq{1}{2}{m} $, and a subsequent transposition implies the identity
\begin{align}
\myunderbrace{
\left( \begin{array}{@{\ }c@{\ \  }c@{\ \ }c@{\ \ }c@{\ }}
\wstart{10} &
\wstart{11} &
\cdots &
\wstart{1,\mym-1}
\\ [2mm]
\wstart{20} &
\wstart{21} &
\cdots &
\wstart{2,\mym-1}
\\ [2mm]
\vdots & \vdots & & \vdots
\\ [2mm]
\wstart{\mym0} &
\wstart{\mym1} &
\cdots &
\wstart{\mym,\mym-1}
\end{array} \right)
}{\dis =: \ T}
\mtilde
= B D,
\label{eq:matrix-t-def}
\end{align}
where $ \mtilde \in \myrnn[\mym] $ 
denotes the transpose of the system matrix in \refeq{start-weights-system-matrix}, and
\begin{align*}
D = \diag\klabi{\tfrac{1}{\myxpo}  \ : \ \myxpo = \myseqq{1}{2}{\mym} }
\in \myrnn[\mym],
\qquad
B = \kla{n^\myxpo}_{n=1,\ldots,\mym \atop \myxpo=1,\ldots,\mym}
\in \myrnn[\mym].
\end{align*}
The matrices $ D, B $ and $ \mtilde $ are regular,
and hence the matrix $ T \in \myrnn[\mym] $ introduced in
\refeq{matrix-t-def} is regular.
In the case $ k \equiv 1 $, the latter matrix coincides with the matrix $ \Sh $.

We now consider the general case for $ k $. We have
$ k\klasm{x,x} = 1 $ and
$ x_n = \xminpl \Landausm{h} $ for $ n \eq \myseqq{1}{2}{\mym-1} $, and thus
$ k\klasm{x_n,x_s} = 1 + \Landausm{h} $ for
$ n = \myseqqsh{1}{2}{\mym} $
and $ s = \myseqqsh{0}{1}{\mym-1} $. We thus have
$ \Sh = T + \Landausm{h} $ for $ h \to 0 $, and 
from this the proposition immediately follows.
\proofend

\bn
Next we consider the approximation properties of the \inapps.
\begin{theorem}
Let the conditions of Assumptions \ref{th:msm-assump} be satisfied.
Then the \inapps 
$ \myseqq{\undelta[0]}{\undelta[1]}{\undelta[\emm-1]} $,
determined by
\refeq{msm-noise-start-values} for $ h $ sufficiently small,
satisfy
\begin{align*}
\max_{n=\myseqq{0}{1}{\mym-1}} \modul{\undelta - u\klasm{x_n} }
= \Landauno{h^\myqq + \delta/h} \as (h,\delta) \to 0.
\end{align*}
\label{th:start-values-error}
\end{theorem}
\proof
It is clear from
\refeq{start-values-system-matrix} and Proposition \ref{th:start-values-system-matrix}
that the \inapps $ \myseqq{\undelta[0]}{\undelta[1]}{\undelta[\mym-1]} $
exist and are uniquely determined for $ h $ sufficiently small.
We have
\begin{align}
h \mysum{s=0}{\mym-1}
\wstart{ns} \cdott k\klasm{x_n,x_s} \myendelta[s]
= \Landauno{h^{\myqq+1} + \delta} \for n = \myseqq{1}{2}{\mym},
\label{eq:start-values-a}
\end{align}
where
\begin{align*}
\enndelta = \undelta[s] - u(\xs), \quad s=\myseqq{0}{1}{\mym-1},
\end{align*}
denote the approximation errors.
This follows
from the considerations in Example \ref{th:interpolatory-start-weights},
with the notation $ r = n $ and for
$ \myfun(y) =  k\klasm{\xn,y} u\klasm{y} $ for $ \xmin \le y \le \xm $. 
A matrix-vector formulation of \refeq{start-values-a}
yields
$ h \Sh \Delta_h^\delta = \Landauno{h^{\myqq+1} + \delta} $ as $ h \to 0 $,
with $ \Delta_h^\delta \defeq (\myendelta[0], \myendelta[1], \ldots, \myendelta[\mym-1])^\tee
\in \reza^\mym $,
and with the matrix $ \Sh $ from \refeq{start-values-system-matrix}.
According to Proposition \ref{th:start-values-system-matrix}, this matrix $ \Sh $
is regular for sufficiently small values of $ h $, and
$ \maxnorm{\Shinv} = \Landausm{1} $ as $ h \to 0 $.
From this the statement of the theorem follows.
\proofend
\subsection{The main result}
We next present the main result of this paper which extends the results by Wolkenfelt (\mynocitea{Wolkenfelt}{79}, \mynocitea{Wolkenfelt}{81}) to the case of perturbed \rhss.
\begin{theorem}
Let the conditions of Assumption \ref{th:msm-assump} be 
satisfied, and let the approximations $ \myseqq{\undelta[0]}{\undelta[1]}{\undelta[\nmax-\mu]} $
be determined by Algorithm \ref{th:msm-volterra-scheme}, for $ h $ sufficiently small.
Then the following error estimate holds,
\begin{align}
\max_{n=\myseqq{0}{1}{\nmax-\mu}} 
\modul{\undelta - u(\xn)}
= \Landauno{h^\myqq + \delta/h} \as \hdeltatonull.
\label{eq:main-msm}
\end{align}
\label{th:main-msm}
\end{theorem}
\begin{proof} 
The \inapp errors are already covered by Theorem \ref{th:start-values-error}, so it remains to estimate the error $ \undelta - u(\xn) $
for $ n=\myseqq{\emm}{\emm+1}{\nmax-\mu} $. 
For this we may assume $ N \ge \mym+\mu $, since otherwise nothing is to be done for.

(1) In a first step we observe
that the following system of error equations holds: 
\begin{align}
h \mysum{s=\mym}{n-\mu} \mygamma_{n-\mu-s} k\klasm{\xn,\xs} 
\enndelta[s]
= \globerr
+ \Landauno{h^{\myqq+1} + \delta}
\for n = \myseqqsh{\mym+\mu}{\mym+\mu+1}{\nmax},
\label{eq:main-msm-a}
\end{align}
uniformly in $ n $, where
\begin{align}
\enndelta &= \undelta[s] - u(\xs), \quad
s=\myseqqsh{\mym}{\mym+1}{\nmax-\mu},
\nonumber \\
\globerr &= 
\sum_{s=\mu}^{n-m} \alpinv_{n-m-s} \myfunspectb{\xn,\xs},
\quad n=\myseqqsh{\mym+\mu}{\mym+\mu1}{\nmax}.
\label{eq:globerr-def}
\end{align}
Furthermore,
\begin{align}
\myfunspectb{x,y} := 
\locerr{z \mapsto k(x,z)u(z),y, h },
\qquad \xmin \le y \le x-mh, \quad \xmin < x \le \xmax,
\label{eq:truncation-ku}
\end{align}
denotes the truncation error corresponding to the function $ \myfun(y)= k\klasm{x,y} u\klasm{y}, \ \xmin \le y \le x $.
The error representation \refeq{main-msm-a} follows by considering the difference of 
the representation \refeq{msm-noise-proofrep} on one side and
the representations in Corollary
\ref{th:ode_msm_representation} on the other side.
We have taken \refeq{start-values-a} and 
$ \mysumtxt{s=0}{\mu-1} \alpinv_{n-m-s} \myfunspectb{\xn,\xs}
= \Landauno{h^{\myqq+1}} $ into consideration here. 
This allows to start summation in \refeq{globerr-def}
with $ s = \mu $.

(2) We next consider a \matvecform of \refeq{main-msm-a}. 
As a preparation we introduce the notation
\begin{align}
\nmaxpmu \defeq \nmax-\mym-\mu+1
\label{eq:nmaxpmu-def}
\end{align}
and consider the system matrix $ \Ah \in \myrnn[\nmaxpmu] $ given by
\begin{align}
\Ah =
\left(
\begin{array}{@{\quadti}c@{\quadti}c@{\quadti}c@{\quadti}c@{\quadti}c@{\quadti}}
\mygammanmupur \mykkom{\mym+\mu}{\mym} & 0 & \cdots & \cdots & 0 \\[4mm]
\mygammanmu[1] \mykkom{\mym+\mu+1}{\mym} & 
\mygammanmupur \mykkom{\mym+\mu+1}{\mym+1} & 
\ddots & & 0 \\[4mm]
\vdots & 
\mygammanmu[1] \mykkom{\mym+\mu+2}{\mym+1} & \ddots  & \ddots &  \vdots \\[4mm]
\vdots & & \ddots  & \ddots & 0 \\[4mm]
\mygamman[\nmax-\mym-\mu] \myk{\nmax}{\mym} & \cdots & \cdots & 
\mygammanmu[1] \mykkom{\nmax}{\nmax-\mu-1} &
\mygammanmupur \mykkom{\nmax}{\nmax-\mu}
\end{array} \right),
\label{eq:ah-def}
\end{align}
with the notation
\begin{align*}
\myk{n}{s} \eq k\klasm{\xn,\xs}
\for \mym \le s \le n-\mu,
\quad
\mym+\mu \le n \le \nmax.
\end{align*}
In addition we consider the vectors 
\begin{align}
\Ehdelta &= \kla{\enndelta[s]}_{s=\mym,\ldots,\nmax-\mu},
\quad
\rh = \kla{\globerr}_{n=\mym+\mu,\ldots,\nmax}.
\label{eq:main-msm-b-1}
\end{align}
Using these notations, the linear system of equations \refeq{main-msm-a}
obviously takes the form
\begin{align}
h \Ah \Ehdelta  = \rh \plus \Fh, 
\withsome \Fh \in \reza^{\nmaxpmu}, \ \maxnorm{\Fh} = \Landauno{h^{\myqq+1}+\delta},
\label{eq:main-msm-b-2}
\end{align}
where $ \maxnorm{\cdot} $ denotes the maximum norm on $ \reza^{\nmaxpmu} $.

(3) For a further treatment of the identity \refeq{main-msm-b-2},
let the matrices 
$ \Wh \in \myrnn[\nmaxpmu] $
and its inverse $ \Whinv \in \myrnn[\nmaxpmu] $
be given by
\feldstretch{1.7}
\begin{align}
\Wh = 
\left(
\begin{array}{c@{\quadsm}c@{\quadsm}c@{\quadsm}c}
\gamma_0 & 0 & \cdots &  0 \\
\gamma_{1} & \gamma_{0} & \ddots & \vdots \\ 
\vdots & \ddots & \ddots & 0 \\ 
\gamma_{\nmax-\mym-\mu} & \cdots & \gamma_{1} & \gamma_{0}
\end{array} \right) \cdott ,
\quad
\Whinv = 
\left(
\begin{array}{c@{\quadsm}c@{\quadsm}c@{\quadsm}c}
\mygammainvnmupur & 0 & \cdots &  0 \\
\mygammainvnmu[1] & \mygammainvnmupur & \ddots & \vdots \\ 
\vdots & \ddots & \ddots & 0 \\ 
\mygammainvnmu[\nmax-\mym-\mu] & \cdots & \mygammainvnmu[1] & \mygammainvnmupur 
\end{array} \right) \cdott .
\label{eq:Wh-def}
\end{align}
We next show that
\begin{align}
\maxnorm{\Whinv } = \Landausm{1},
\qquad
\maxnorm{\Ah^{-1} \Wh} 
= \Landauno{1},
\qquad
\maxnorm{\Ah^{-1} } = \Landausm{1} \as h \to 0,
\label{eq:main-msm-c}
\end{align}
where $ \maxnorm{\cdot} $ denotes the matrix norm induced by the maximum vector norm on $ \reza^{\nmaxpmu} $. 
In fact, the estimate $ \maxnorm{\Whinv } = \Landausm{1} $ as $ h \to 0 $
follows immediately from the exponential decay of the coefficients of the \inverse of the \genfunc $ \mygamma(\xi) $, 
\cf 
part (a) of Remark \ref{th:strong-root-cond-remarks}.
For the proof of the second statement in \refeq{main-msm-c}, below it will be shown that 
the matrix $ \Whinv \Ah $ can be written in the form
\begin{align}
\Whinv \Ah = I_h  + L_h, 
\label{eq:fredholm-rep}
\end{align}
where $ I_h \in \myrnn[\nmaxpmu] $ denotes the identity matrix,  
and $ L_h = \kla{\mykap_{nj}(h)} \in \myrnn[\nmaxpmu] $
denotes some lower triangular matrix which satisfies
$ \max_{0 \le j \le n \le \nmax-\emm-\mu} {\modul{\mykap_{nj}(h)}} = \Landauno{h} $ as $ h \to 0 $.
The representation \refeq{fredholm-rep} shows that the matrix $\Whinv \Ah $ is nonsingular for $ h $ small enough, and the discrete version of Gronwall's inequality 
then yields $ \maxnorm{\Ah^{-1} \Wh} = \Landauno{1} $ as $ h \to 0 $.
The third estimate in \refeq{main-msm-c} follows immediately from the other two 
estimates considered in \refeq{main-msm-c}.

\bn
In the sequel it will be shown that the representation
\refeq{fredholm-rep} is valid, and for this purpose
we consider the lower triangular matrix
\begin{align}
\Whinv \Ah =\kla{b_{nj}} \in \myrnn[\nmaxpmu]
\label{eq:bnj-def}
\end{align}
in more detail.
In fact, we have
for $ 0 \le j < n \le \nmaxpmumo $
\begin{align*}
& b_{nj}
=
\mysum{\myl=j}{n}
\mygammainvnmu[n-\myl] \mygammanmu[\myl-j] k\kla{x_{\mym+\mu+\myl},x_{\emm+j}}
=
\mysum{\myl=0}{n-j}
\mygammainvnmu[n-j-\myl] \mygammanmu[\myl] k\kla{x_{\mym+\mu+\myl+j},x_{\emm+j}}
\\
& =
k\kla{x_{\mym+\mu+n},x_{\emm+j}}
\myoverbrace{\mysum{\myl=0}{n-j}
\mygammainvnmu[n-j-\myl] \mygammanmu[\myl]}{\dis = \  0 }
\\
& \qquad
+ \mysum{\myl=0}{n-j-1}
\eckklabi{\mygammainvnmu[n-j-\myl] \mygammanmu[\myl]
\klabi{k\kla{x_{\mym+\mu+\myl+j},x_{\emm+j}}
- k\kla{x_{\mym+\mu+n},x_{\emm+j}} }}.
\end{align*}
Thus we have
\begin{align}
\modul{b_{nj}}
& =
\Landaula{h
\myunderbrace{\mysum{\myl=0}{n-j-1}
\modul{\mygammainvn[n-j-\myl]} \modul{\mygammanmu[\myl]} \kla{n-j-\myl}
}{\dis \idstar \Landauno{1}}
}
=
\Landauno{h}
\for \ 0 \le j < n \le \nmaxpmumo
\label{eq:convergence-theorem-i}
\end{align}
uniformly \wrt $ j $ and $ n $,
where \klasmsh{*} follows immediately from 
\refeq{gamman-bounded} and
the end of part (a) of Remark \ref{th:strong-root-cond-remarks}.
Moreover we have
\begin{align}
b_{nn} =
\mygammainvnmupur k\klasm{\xn[n+\mym+\mu],x_{n+\emm}} \mygammanmupur
= 1 + \Landauno{h} \for n = \myseqq{0}{1}{\nmaxpmumo},
\label{eq:convergence-theorem-j}
\end{align}
which follows from the identities $ \mygammainvnmupur = 1/\mygammanmupur $ and
$ k(x,x) \equiv 1 $, \cf \ref{item:assump-k=1} in Assumption~\ref{th:msm-assump}.
The statements \refeq{convergence-theorem-i} and \refeq{convergence-theorem-j}
show that the lower triangular matrix $ \Whinv \Ah $ in fact can be written as
in \refeq{fredholm-rep}.

(4) 
We still have to take a closer look at the
vector $ \rh \in \reza^{\nmaxpmu} $
considered in \refeq{main-msm-b-1}. It can be written as follows,
\begin{align}
\rh = \Bh \erra_h,
\label{eq:main-msm-f}
\end{align}
where
$ \Bh \in \myrnn[\nmaxpmu] $ is the following matrix,
\begin{align*}
\left(
\begin{array}{@{\hspace{0mm}}c@{\hspace{2mm}}c@{\hspace{-8mm}}c@{\hspace{2mm}}c@{\hspace{0mm}}}
\alpinv_0 \psitkomb{\mym+\mu}{\mu} & 0 & \cdots & 0 \\[6mm]
\alpinv_1 \psitkomb{\mym+\mu+1}{\mu} & 
\alpinv_0 \psitkomb{\mym+\mu+1}{\mu+1} &  \ddots & \vdots \\[6mm]
\vdots & 
\alpinv_1 \psitkomb{\mym+\mu+2}{\mu+1} & \ddots  &  \vdots \\[6mm]
\vdots & & \ddots & 0 \\[5mm]
\alpinv_{\nmax-\mym-\mu} \psitkomb{\nmax}{\mu} & \cdots & 
\alpinv_1 \psitkomb{\nmax}{\nmax-\emm-1} &
\alpinv_0 \psitkomb{\nmax}{\nmax-\emm}
\end{array} \right),
\end{align*}
and
$ \erra_h = (\myseqq{1}{1}{1}) \in \reza^{\nmaxpmu} $.
The representations \refeq{main-msm-b-2}
and \refeq{main-msm-f} give
$ h \Ah \Ehdelta 
= \Bh  \erra_h + \Fh $,
and due to \refeq{main-msm-c} it remains to show that
\begin{align}
\maxnorm{\Whinv \Bh} = \Landauno{h^{\myqq+1}}
\as h \to 0
\label{eq:main-msm-h-2}
\end{align}
holds.  \Ftp we introduce the notation
\begin{align*}
\Uh = \left(
\begin{array}{c@{\hspace{1mm}}c@{\hspace{1mm}}c@{\hspace{1mm}}c@{\hspace{1mm}}c@{\hspace{2mm}}c}
\alpha_0 & 0 & \cdots & \cdots & \cdots  & 0 \\[-2mm]
\vdots & \ddots & \ddots & \ddots &  & \vdots \\[-2mm]
\alpha_m & & \ddots & \ddots & &  \vdots \\[-2mm]
0 & \ddots & & \ddots & \ddots &\vdots \\[-2mm]
\vdots & \ddots & \ddots & \ddots & \ddots & 0 \\[-2mm]
0 & \cdots & 0 &  \alpha_m & \cdots & \alpha_0 
\end{array} \right), 
\quad
\Uhinv = \left(
\begin{array}{c@{\hspace{1mm}}c@{\hspace{1mm}}c@{\hspace{1mm}}c}
\alpinv_0 & 0 & \cdots & 0 \\
\alpinv_1 & \alpinv_0 & \ddots & \vdots \\ 
\vdots & \ddots & \ddots &  0 \\
\alpinv_{N-m-\mu} & \cdots & \alpinv_1  & \alpinv_0  \\
\end{array} \right)
\in \reza^{N_1 \times N_1},
\\
\Vh = \left(
\begin{array}{c@{\hspace{1mm}}c@{\hspace{1mm}}c@{\hspace{1mm}}c@{\hspace{1mm}}c@{\hspace{2mm}}c}
\beta_0 & 0 & \cdots & \cdots & \cdots  & 0 \\[-2mm]
\vdots & \ddots & \ddots & \ddots &  & \vdots \\[-2mm]
\beta_{m-\mu} & & \ddots & \ddots & &  \vdots \\[-2mm]
0 & \ddots & & \ddots & \ddots &\vdots \\[-2mm]
\vdots & \ddots & \ddots & \ddots & \ddots & 0 \\[-2mm]
0 & \cdots & 0 &  \beta_{m-\mu} & \cdots & \beta_0
\end{array} \right), 
\quad
\Vhinv = \left(
\begin{array}{c@{\hspace{1mm}}c@{\hspace{1mm}}c@{\hspace{1mm}}c}
\betinv_0 & 0 & \cdots & 0 \\
\betinv_1 & \betinv_0 & \ddots & \vdots \\ 
\vdots & \ddots & \ddots &  0 \\
\betinv_{N-m-\mu} & \cdots & \betinv_1  & \betinv_0  \\
\end{array} \right)
\in \reza^{N_1 \times N_1},
\end{align*}
and observe that
\begin{align}
\Wh = \Vh \Uhinv, \quad 
\Whinv = \Vhinv \Uh,
\label{eq:main-msm-h-3}
\end{align}
holds.
From the fact that the second characteristic polynomial
(see \refeq{sigma-strong-root-condition})
is a \schurpolynomial it follows
\begin{align}
\maxnorm{\Vhinv} = \Landauno{1} \quad \textup{as } h \to 0.
\label{eq:main-msm-h-4}
\end{align}
In the sequel we consider the
lower triangular matrix $ \Uh \Bh $ in more detail. It can be written as follows,
$ \Uh \Bh = \Mhs + \Chs $ with the diagonal matrix
$ \Mhs = \text{diag} \kla{\psitkomb{m+n}{n} : n = \myseqq{\mu}{\mu+1}{N-m} } $,
with $ \maxnorm{\Mhs} = \Landauno{h^{\myqq+1}} $ as $ h \to 0 $. In addition,
$ \Chs = \kla{c_{nj}(h)} \in \myrnn[\nmaxpmu] $
denotes some strictly lower triangular matrix
with $ \max_{0\le j < \n \le \nmaxpmumo} {\modul{c_{nj}(h)}}
= \Landausm{h^{\myqq+2}} $. 
See the third part of this proof for similar results \wrt the matrix $ \Whinv \Ah $.
Here we additionally use the mean value theorem
 with respect to the first variable of $ g $ and the fact that the local truncation error $g$ 
defined in \refeq{truncation-ku} satisfies
\begin{align*}
\tfrac{\partial }{\partial x} \myfunspectb{x,y} = 
\locerrbi{ z \mapsto \tfrac{\partial }{\partial x}  k(x,z)u(z),y, h }
= \Landauno{h^{\myqq+1}}
\end{align*}
uniformly for $ \xmin \le y \le x -m h $ and $ \xmin < x \le \xmax $.

This in particular means
$ \maxnorm{\Uh \Bh} = \Landauno{h^{\myqq+1}} $ as $ h \to 0 $,
and this together with
\refeq{main-msm-h-3} and
\refeq{main-msm-h-4}
implies
 \refeq{main-msm-h-2}.

The statement of the theorem now follows easily from 
the error representation \refeq{main-msm-a} and its matrix version
\refeq{main-msm-b-1}, \refeq{main-msm-b-2}, from
the stability estimates in \refeq{main-msm-c}, and from the considerations in part (4) of this proof.
\end{proof} 
\begin{remark}
The stability analysis presented in the third part of the proof of
Theorem \ref{th:main-msm} uses techniques similar to those used in
\mycitea{Eggermont}{81}; see also 
\mycitea{Lubich}{87} as well as
\mynocitea{Plato}{05} and \mynocitea{Plato}{12}.
\end{remark}

\bn
In the sequel, for \stepsizes $ h = h(\delta) = \tfrac{\diffxmaxxmin}{N} $,
with $ N =  N(\delta) $, with a slight abuse of notation  
we write
$ h \sim \delta^\beta $ as $ \delta \to 0 $, 
if there exist real constants $ c_2 \ge c_1 > 0 $ such that
$ c_1 h \le \delta^{\beta} \le c_2 h $ holds for $ \delta \to 0 $.
As an immediate consequence of Theorem \ref{th:main-msm}
we obtain the following main result of this paper.
\begin{corollary}
Let Assumption \ref{th:msm-assump} be satisfied.
For $ h = h(\delta) \sim \delta^{1/(\myqq+1)}$ we have
\begin{align*}
\max_{n=\myseqq{0}{1}{\nmax-\mu}} \modul{\undelta \minus u(\xn)}
= \Landauno{\delta^{\lfrac{\myqq}{(\myqq+1)}}}
\as \delta \to 0,
\end{align*}
where the approximations $ \myseqq{\undelta[0]}{\undelta[1]}{\undelta[\nmax-\mu]} $
are determined by Algorithm \ref{th:msm-volterra-scheme}.
\label{th:msm-apriori-choice}
\end{corollary}
We conclude this section with some remarks.
\begin{remark}
\begin{myenumerate}
\item
Assumption \ref{th:msm-assump} and Corollary \ref{th:msm-apriori-choice}
imply that the order of the method should be chosen as large as possible to allow best possible estimates for a wide range of smoothness degrees of solutions.
Note that, for $ m $ fixed, both the computational complexity and the number of function evaluations for the implementation of Algorithm
\ref{th:msm-volterra-scheme} are $ \Landau(N^2) $ as $ N \to \infty $. 
Thus the number of steps $ m $ in the considered multistep method has no impact here.

\item
For results on the regularization properties of the composite midpoint rule, see 
e.\,g.~\mycitea{Apartsin}{81} or
 \mycitea{Kaltenbacher}{10}.
For other special regularization methods for the approximate solution of
Volterra integral equations of the first kind with smooth kernels and perturbed 
\rhss, see \eg
\mycitea{Lamm}{00}.
\end{myenumerate}
\label{th:main-remark}
\end{remark}
\section{The balancing principle}
\subsection{Preparations}
The a priori choice of the \stepsize $ h $ considered 
in Corollary \ref{th:msm-apriori-choice}
requires knowledge of the smoothness of the exact solution $ u: \interval{\xmin}{\xmax} \to \reza $.
The balancing principle as an a posteriori strategy for choosing $ h $
has no such requirement and thus seems to be an interesting alternative. 
Its implementation, however, requires 
a determination of the coefficient of the error propagation term $ \lfrac{\delta}{h} $
that appears in the basic error estimate \refeq{main-msm}.
This is the subject of the following proposition.
\begin{proposition}
Under the conditions of Assumption \ref{th:msm-assump}
we have
\begin{align}
\max_{n=\myseqq{0}{1}{\nmax-\mu}} 
\modul{\undelta - u(\xn)}
\le  C_1 h^\myqq + C_2 \tfrac{\delta}{h} \quad \text{ for } 0 < h \le \myhbound,
\label{eq:msm-estimate-with-constant}
\end{align}
where $ C_1 $ and $ C_2 $ denote some constants chosen independently of $ h $, and $ \myhbound $ is chosen sufficiently small.
The constant $ C_2 $ may be chosen as follows:
\begin{align*}
C_2 &=
\max\Big\{C_{2a}, C_{2b} \big(1+C_{2a} \maxnorm{k}
\max_{m+\mu \le n \le \nmax} \mysum{s=0}{m-1} \modul{\myomega_{ns} }\big)
\Big\}, 
\ \textup{ where} 
\\ 
C_{2a} &= (1+L)\maxnorm{T^{-1}}, 
\quad
C_{2b} = 
(1+\muL) 
\Big(\mysum{s=0}{\infty} \vert \gaminv_{s} \vert\Big)
\exp\big((1+\muL) C_3 L\diffxmaxxminkla\big)
\\[-0mm]
& \qquad 
\ \textup{ with } \ 
C_{3} = \Big\{\sup_{r \ge 0} \vert \gamma_{r} \vert \Big\}
\mysum{s=1}{\infty} \vert \gaminv_{s} \vert s,
\end{align*}
where the notation $ \maxnorm{k} = \max_{(x,y) \in E} \modul{k(x,y)} $ is used,
and $ L \ge 0 $ denotes a Lipschitz constant of the kernel $ k $ with respect to the first variable. 
In addition, for the definition of
the sequence $ (\gaminv_s) $
and the matrix $ T $,
see \refeq{betinv-gaminv-def} and \refeq{matrix-t-def}, respectively.

Moreover,
$ \myhbound $ in \refeq{msm-estimate-with-constant} can be chosen as follows,
$ \myhbound = \min\{\tfrac{1}{m(1+L)\cond_\infty(T)}, \hmax\} $,
where $ \hmax $ is taken from Assumption \ref{th:msm-assump}.
In the special case $ k \equiv 1 $, 
the estimate \refeq{msm-estimate-with-constant}
holds with $ \myhbound = \hmax $.
\label{th:msm-estimate-with-constant}
\end{proposition}
\proof
Let $ \enndelta = \undelta[s] - u(\xs) $ for $ s=\myseqqsh{0}{1}{\nmax-\mu} $.
We first consider the starting error.
A closer look at the proof of Theorem \ref{th:start-values-error}
shows that
\begin{align}
\max_{ s=\myseqqsh{0}{1}{m-1}}\modul{\enndelta}
\le  
\maxnorm{\Sh^{-1}}(C_4 h^\myqq + \tfrac{\delta}{h})
\ \  \text{for} \ h > 0, 
\label{eq:msm-estimate-with-constant-a}
\end{align}
where $ \Sh $ denotes the system matrix considered in \refeq{start-values-system-matrix} and \refeq{start-values-a}, and $ h $ is chosen so small (details are given below)
such that the inverse matrix of $ \Sh $
exists. In addition, $ C_4 $ denotes some constant that may be chosen independently of $ h $.
So we need to estimate $ \maxnorm{\Sh^{-1}} $ which is done below.
First we consider the error of the present multistep scheme.
A closer look at the reasoning of \refeq{main-msm-a} shows that
\begin{align*}
h \mysum{s=\mym}{n-\mu} \mygamma_{n-\mu-s} k\klasm{\xn,\xs} \enndelta[s]
=
\fndelta-f(\xn) +  \globerr
- h \mysum{s=0}{m-1} \myomega_{ns} k\klasm{\xn,\xs} \enndelta[s]
+ \Landauno{h^{\myqq+1}} 
\end{align*}
holds uniformly for $ n = \myseqqsh{\mym+\mu}{\mym+\mu+1}{\nmax} $, where $ \gamma_0, \gamma_1, \ldots $ are given by \refeq{gamma-def}.
Representation \refeq{main-msm-b-2} in the proof of Theorem \ref{th:main-msm} thus can be written as
\begin{align}
h \Ah \Ehdelta 
= 
\rh \plus \Fha \plus \Fhb, 
\withsome \Fha \in \reza^{\nmaxpmu}, \ \maxnorm{\Fha} = \Landauno{h^{\myqq+1}},
\label{eq:msm-estimate-with-constant-ab}
\end{align}
and some vector $ \Fhb \in \reza^{\nmaxpmu} $ with
\begin{align}
\maxnorm{\Fhb} \le \delta + h 
\maxnorm{k}
\Big\{\max_{ m+\mu \le n \le \nmax } \mysum{s=0}{m-1} \modul{\myomega_{ns} } \Big\} 
\max_{0 \le s \le m-1} \modul{e_s^\delta}.
\label{eq:msm-estimate-with-constant-b}
\end{align}
So in view of
\refeq{msm-estimate-with-constant-a}--\refeq{msm-estimate-with-constant-b}
we need to provide upper bounds for
$ \maxnorm{\Sh^{-1}} $ and 
$ \maxnorm{\Ah^{-1}} $.
For this purpose let $ L \ge 0 $ denote a Lipschitz constant of the kernel $ k $ 
with respect to the first variable, i.e.,
\begin{align*}
\vert k(x_1,y) - k(x_2,y) \vert \le L 
\vert x_1 - x_2 \vert \for
(x_1,y), (x_2, y) \in E,
\end{align*}
where the set $ E $ is introduced in Assumption \ref{th:msm-assump}.
Then the matrix $ \Sh,  h \le \hmax, $ can be written in the form
$ \Sh = T + F_h $, where the perturbation matrix $ F_h \in \myrnn[m] $ satisfies
$ \maxnorm{F_h} \le \maxnorm{T} m L h $.
It then follows from standard perturbation results for matrices that
\begin{align}
\maxnorm{\Sh^{-1}} \le (1+L) \maxnorm{T^{-1}} = C_{2a}
\for
0 < h \le \frac{1}{m(1+L) \cond_\infty(T)},
\label{eq:Shinv-bound}
\end{align}
where $ \cond_\infty(T) = \maxnorm{T} \maxnorm{T^{-1}} $,
and the upper bound for $ h $ in \refeq{Shinv-bound}
can be ignored
if $ L = 0 $.

For the estimation of $ \maxnorm{\Ah^{-1}} $ we have to take a closer look at part (3) of the proof of Theorem \ref{th:main-msm}. We obviously have
$ \maxnorm{\Whinv} \le \mysumtxt{s=0}{\infty} \vert \gaminv_s \vert $
for $ h > 0 $, and we next estimate the entries 
of $ \Whinv \Ah = (b_{nj}) $ (cf.~\refeq{bnj-def}).
Continuing from \refeq{convergence-theorem-j}
gives $ \vert b_{nn} \vert \ge 1- L \vert \xn[n+\mym+\mu]-x_{n+\emm} \vert
= 1 - \muL h \ge \tfrac{1}{1+\muL} $ for $ h \le \tfrac{1}{1+\muL} $.
Proceeding from \refeq{convergence-theorem-i}
yields $ \vert b_{nj} \vert \le C_3 L h $ for $ j < n $, where the constant $ C_3 $ is chosen as in the statement of the proposition. An application of the discrete version of Gronwall's lemma now results in
\begin{align*}
\maxnorm{\Ah^{-1}} & \le 
\maxnorm{(\Whinv \Ah)^{-1}} \maxnorm{\Whinv } 
\\
& \le (1+\mu L)
\big(\mysum{s=0}{\infty} \vert \gaminv_s \vert \big)
\exp((1+\muL) C_3 L \diffxmaxxminkla)
= C_{2b} \for 0 < h \le \tfrac{1}{1+\muL},
\end{align*}
where the considered upper bound for $ h $ can be ignored
if $ \mu = 0 $ or $ L = 0 $ holds. Note also that this upper bound for $ h $ is not smaller than the upper bound for $ h $ given in \refeq{Shinv-bound} which justifies the
definition of $ \myhbound $ given in the proposition.
\proofendspruch

\subsection{Implementation of the balancing principle}
In the sequel we assume that the conditions of 
Assumption \ref{th:msm-assump} are satisfied.
It is convenient to introduce new notation for the set of \gridpoints
and for the approximations generated by the considered \msm
to indicate dependence on the \stepsize $ h $:
\begin{align}
& \Ih{h} = \inset{\xn = \xminpl nh \mid n = \myseqq{0}{1}{N-\mu} }, 
\text{ where } h = \tfrac{\diffxmaxxmin}{N},  \ N \ge \Nmin, 
\nonumber \\
& \udeltas{\cdot}{h}: \Ih{h} \to \reza, \quad \xn \mapsto \undelta.
\label{eq:udeltah-def}
\end{align}
In the sequel we consider the following sequence of geometrically increasing \stepsizes: 
\begin{align}
\hs & = \tfrac{\diffxmaxxmin}{\Ns}, \ \Ns = \Nsmin 2^{\kappa (\squer-s)} \for s=\myseqq{0}{1}{\squer},
\label{eq:hs-def}
\end{align}
where $ \squer = \squer(\delta) \ge 0 $ and $ \Nsmin = \Nsmin(\delta) \ge 1 $ are some integers that may depend on $ \delta $, and $ \kappa \ge 1 $ is some fixed integer.
The set of those \stepsizes will be denoted by $ \Hset $, \ie
\begin{align*}
\Hset & = \inset{\hs[0] < \hs[1] < \cdots < \hs[\squer]}.
\end{align*}
Note that due to the special form of the \stepsizes we have 
\begin{align*}
\Ih{\hs[\squer]} \subset \Ih{\hs[\squer-1]} \subset \cdots \subset \Ih{\hs[0]}.
\end{align*}
In the sequel we assume that $ \squer \ge 0 $
and 
$ \Nsmin \ge 1 $
are chosen so 
that the \stepsizes 
$ \hs[0] $
and $ \hs[\squer] $
are respectively sufficiently small
and sufficiently large.
More precisely, we assume the following:
\begin{align}
\hs[0] \le c_*\delta^{1/2}, 
\qquad
c_{**}\delta^{1/(\myp+1)} \le \hs[\squer] \le \myhbound
\qquad (0 < \delta \le \delta_0),
\label{eq:hnull-gross-hsquer-klein}
\end{align}
where $ c_*, \, c_{**} $ and $  \delta_0 > 0 $ denote some constants,
and $ \myhbound $ is chosen as in Proposition~\ref{th:msm-estimate-with-constant}.
In addition, $ c_{**} $ is chosen sufficiently small such an $\hs[\squer] $ satisfying \refeq{hnull-gross-hsquer-klein} exists.

We consider the following a posteriori choice of the \stepsize $ h = \hdelta $:
\begin{align}
\hdelta = \max \Hdelta,
\ \textup{where } \Hdelta \defeq \inset{ \hstar \in \Hset :
\textup{ for } h, \hstst \in \Hset \textup{ with } h < \hstst \le \hstar \textup{ we have }  &
\nonumber \\
\max_{y \in \Ih{\hstst}} \vert \udeltas{y}{\hstst} - \udeltas{y}{h} \vert
\le \myeta \tfrac{\delta}{h}}, &
\label{eq:hdelta-def}
\end{align}
where $ \myeta > 2 C_2 $ holds, with
$ C_2 $ chosen as in Proposition \ref{th:msm-estimate-with-constant}.
Note that by definition we have $ \hs[0] = \min \Hset \in \Hdelta $ 
so that $ \Hdelta \neq \varnothing $, and thus
$ \hdelta $ in \refeq{hdelta-def} is well-defined.
The adaptive choice of the \stepsize given by \refeq{hdelta-def}
is in fact a balancing principle.
For a general introduction to this class of a posteriori parameter choice strategies see, \eg 
\myciteatwo{Lepski\u{\i}}{Lepskii[90]},
\myciteatwo{Math\'{e}}{Mathe[06]},
\myciteb{Pereverzev}{Schock}{05}, or
\myciteb{Lu}{Pereverzev}{13}.
\begin{remark}
The strategy \refeq{hdelta-def} is in fact a nonstandard balancing principle.
We recall that the classical balancing principle chooses, in our framework, the maximum from the set $ \Hdeltatwo \defeq \inset{ \hstst \in \Hset :
\vert \udeltas{y}{\hstst} - \udeltas{y}{h} \vert
\le \myeta \tfrac{\delta}{h} \ \text{for} \ y \in \Ih{\hstst}, \ h \in \Hset, \ h < \hstst} $.
The latter maximum may be larger than
$ \hdelta $ introduced in \refeq{hdelta-def}, in general.
In turns out, however, that the \stepsize $ \hdelta $ is sufficiently large to get similar estimates as for the standard balancing principle; see the following theorem for details.

The nonstandard version \refeq{hdelta-def} of the balancing principle is considered for computational reasons: it may require less computational amount than the standard version.
In fact, a possible strategy to determine $ \hdelta $ 
is to verify for $ s = 1,2,\ldots $ whether $ \hs \in \Hdeltatwo $ is satisfied, and this procedure stops if $ \hs \not \in \Hdeltatwo $ holds for the first time, or if $ s = \squer $. In the former case we have $ \hdelta = \hs[s-1] $,
and then there is no need to consider the \stepsizes 
$ \myseqq{\hs[s+1]}{\hs[s+2]}{\hs[\squer]} $.
\end{remark}
We have the following convergence result:
\begin{theorem}
Let Assumption \ref{th:msm-assump} be satisfied, and
let 
$ \udeltas{\cdot}{h} $ and $ \hdelta $ be given by
\refeq{udeltah-def} and \refeq{hdelta-def}, respectively.
Then the following estimates hold,
\begin{align}
\max_{y \in \Ih{\hdelta}} \vert \udeltas{y}{\hdelta} - u(y) \vert
& = \Landauno{\delta^{\myqq/(\myqq+1)}} \as \delta \to 0,
\label{eq:estimate-balancing-principle} \\
\hdelta & \ge C \delta^{1/(\myqq+1)},
\label{eq:estimate-hdelta-balancing-principle}
\end{align}
where $ C > 0 $ denotes some constant which is independent of $ \delta $.
\label{th:balancing-principle-main}
\end{theorem}
\proof 
The proof is a compilation of techniques used, \eg in
\myciteb{Lu}{Pereverzev}{13}, and we thus give a sketch of a proof only.
A basic ingredient in the following analysis is provided by the following estimate,
which follows from
Proposition \ref{th:msm-estimate-with-constant} and
\refeq{hdelta-def}:
\begin{align}
\max_{y \in \Ih{\hdelta}} \vert \udeltas{y}{\hdelta} - u(y) \vert
& \le
\max_{y \in \Ih{\hdelta}} \vert \udeltas{y}{\hdelta} - \udeltas{y}{h} \vert
+
\max_{y \in \Ih{h}} \vert \udeltas{y}{h} - u(y) \vert
\nonumber \\
&\le C_1 h^\myqq + (\myeta + C_2) \frac{\delta}{h}
\foreach h \in \Hset, \ h \le \hdelta.
\label{eq:balancing-principle-main-d}
\end{align}
It now remains to determine some $ h \in \Hset $ with $  h \le \hdelta $ and
$ h \sim \delta^{1/(\myqq+1)} $;
the estimates \refeq{estimate-balancing-principle}--\refeq{estimate-hdelta-balancing-principle}
then easily follow from \refeq{balancing-principle-main-d}.
For this purpose we consider the set
\begin{align*}
\Mdelta \defeq \inset{ h \in \Hset : h^{\myqq+1} \le C_3 \delta}, 
\end{align*}
where $ C_3 > 0 $ is chosen so small such that $ 2(C_1C_3 +C_2) \le \myeta $ holds, 
with $ C_1 $ and $ C_2 $ being chosen as in Proposition \ref{th:msm-estimate-with-constant}. That choice of $ C_3 $ guarantees 
\begin{align*}
\Mdelta \subset \Hdelta
\end{align*}
which is shown in the sequel. For 
this purpose let
$ \hstar \in \Mdelta $ and $ h, \hstst \in \Hset $ with $ h < \hstst \le \hstar $.
We then have
\begin{align*}
\max_{y \in \Ih{\hstst}}
\vert \udeltas{y}{\hstst} - \udeltas{y}{h} \vert
& \le
\max_{y \in \Ih{\hstst}}
 \vert \udeltas{y}{\hstst} - u(y) \vert
+
\max_{y \in \Ih{h}} \vert \udeltas{y}{h} - u(y) \vert
\\
&\le C_1 {\hstst}^{\myqq} +C_2 \frac{\delta}{\hstst}
+ C_1 h^\myqq +C_2 \frac{\delta}{h} 
\le 2(C_1C_3+C_2) \frac{\delta}{h},
\end{align*}
where $ h, \hstst \in \Mdelta $ is taken into account. 
This shows $ \hstar \in \Hdelta $ and completes the proof of
the relation $ \Mdelta \subset \Hdelta $.

We are now in a position to verify \refeq{estimate-balancing-principle}--\refeq{estimate-hdelta-balancing-principle}, and
for this we consider two situations. In the case $ \Mdelta \neq \varnothing $ we define
$ \hdeltapl = \max \Mdelta $ and obtain
\begin{align}
\hdeltapl \le \hdelta, \quad
\hdeltapl \sim \delta^{1/(\myqq+1)},
\label{eq:balancing-principle-main-a}
\end{align}
where we assume that $ \delta \le \delta_0 $ holds.
The first statement in \refeq{balancing-principle-main-a} follows immediately 
from $ \Mdelta \subset \Hdelta $ and the definition 
of $ \hdelta $, see \refeq{hdelta-def}.
The second statement  
in \refeq{balancing-principle-main-a} follows in the case
$ \hdeltapl = \max \Hset $ (which is $ \hs[\squer] $ in fact) from
the second estimate in \refeq{hnull-gross-hsquer-klein}, and 
in the case $ \hdeltapl < \max \Hset $ it follows from
$ 2^\kappa \hdeltapl \in \Hset\backslash\Mdelta $.
Estimate \refeq{estimate-hdelta-balancing-principle} is an immediate consequence of
\refeq{balancing-principle-main-a}, and
estimate \refeq{estimate-balancing-principle}
then follows easily
from estimate \refeq{balancing-principle-main-d}, applied with $ h = \hdeltapl $.

On the other hand, $ \Mdelta = \varnothing $ means $ \min \Hset = \hs[0] \not \in \Mdelta $, and the first estimate in \refeq{hnull-gross-hsquer-klein} then implies
$ \min \Hset \sim \delta^{1/(\myqq+1)} $ for $ 0 < \delta \le \delta_0 $.
This shows \refeq{estimate-hdelta-balancing-principle},
and estimate \refeq{estimate-balancing-principle} follows easily
from \refeq{balancing-principle-main-d},
 applied with $ h = \min \Hset $.
\proofend

\section{Numerical experiments}
\label{num_exps}
As an illustration of the main results considered in Corollary \ref{th:msm-apriori-choice}
and Theorem \ref{th:balancing-principle-main}, we next present the results of numerical experiments for four
\voltinteqs of the form \refeq{volterra-inteq}, treated by different kind of \msms, respectively.

Here are two comments on the first three numerical tests, where a priori choices of the \stepsize are considered in fact:
\begin{mylist_indent}
\item
Numerical experiments on the interval $ \interval{\xmin}{\xmax} = \interval{0}{1} $
are employed for
\stepsizes
$ h = 1/2^\nu $ for $ \nu = \myseqq{5}{6}{12} $, with the exception of the order 4 BDF method. In the latter
method, the influence of rounding errors becomes clearly visible for 
$ \nu \ge 10 $.

\item
For each considered \stepsize $ h $ and each considered 
\msm with maximal order $ \myp $,
we consider \refeq{volterra-inteq} with some function $ u \in \UCpl{\myp-1}[0,1] $, and the noise level $ \delta = h^{1/(\myp+1)} $ is considered.
\end{mylist_indent}

In all numerical experiments, the perturbations are of the 
form $ \fndelta = \fxn + \Deltap_n $
with uniformly distributed random values $ \Deltap_n $ with
$ \modul{\Deltap_n} \le \delta $.

\begin{example}
\label{th:num_example_1}
First we consider the repeated midpoint rule which in fact coincides with the 2-step Nystr\"om method (see Example \ref{th:msm-examples}).
In the formulation \refeq{multstep-def}, this quadrature method reads as follows,
$ \myphi_{r+2} - \myphi_{r} = 2h \myfun_{r+1} $ for $ r = \myseqq{0}{1}{n-2} $.
This method is applied to the following linear Volterra integral equation 
of the first kind,
\begin{align}
\ints{0}{x}
{\cos\kla{x-y} u\klasm{y} }{dy}
=
\sin x =: f(x) \for \intervalarg{x}{0}{1},
\label{eq:numeric-a}
\end{align}
with exact solution
$ u\klasm{y} = 1 $
for $ \intervalarg{y}{0}{1} $.
The conditions
of Assumption \ref{th:msm-assump} are satisfied with $ \emm = \myp = \myqq = 2 $.
The numerical results are shown in Table \ref{tab:num1}. There,
$ \maxnorm{f} $ denotes the maximum norm of the function $ f $.
All numerical experiments are employed using the program system \octave \kla{http://www.octave.org}.
\begin{table}[h!]
\hfill
\begin{tabular}{|| r | c |@{\hspace{5mm} } l | c | c ||} 
\hline
\hline
$ N $  
& $ \delta $
& $ 100 \myast \delta/\maxnorm{f} $
& $ \ \max_{n} \modul{\undelta - u\klasm{\xn}} \ $
& $ \ \max_{n} \modul{\undelta - u\klasm{\xn}} \ / \delta^{2/3} \ $
\\ \hline \hline
$  32 $ & $3.1 \myast 10^{-5}$ & $3.70 \myast 10^{-3}$ & $1.05 \myast 10^{-3}$ & $ 1.07$ \\
$  64 $ & $3.8 \myast 10^{-6}$ & $4.58 \myast 10^{-4}$ & $3.09 \myast 10^{-4}$ & $1.27$ \\
$  128 $ & $4.8 \myast 10^{-7}$ & $5.70 \myast 10^{-5}$ & $6.56 \myast 10^{-5}$ & $1.08$ \\
$  256 $ & $6.0 \myast 10^{-8}$ & $7.10 \myast 10^{-6}$ & $1.69 \myast 10^{-5}$ & $1.11$ \\
$  512 $ & $7.5 \myast 10^{-9}$ & $8.87 \myast 10^{-7}$ & $7.25 \myast 10^{-6}$ & $1.90$ \\
$ 1024 $ & $9.3 \myast 10^{-10}$ & $1.11 \myast 10^{-7}$ & $1.09 \myast 10^{-6}$ & $1.14$ \\
$ 2048 $ & $1.2 \myast 10^{-10}$ & $1.38 \myast 10^{-8}$ & $2.71 \myast 10^{-7}$ & $1.14$ \\
$ 4096 $ & $1.5 \myast 10^{-11}$ & $1.73 \myast 10^{-9}$ & $6.71 \myast 10^{-8}$ & $1.13$ \\
\hline
\hline
\end{tabular}
\hfill 
\caption{Numerical results of the repeated midpoint rule applied to equation \refeq{numeric-a}}
\label{tab:num1}
\end{table} 
\end{example}
\begin{example}
\label{th:num_example_2}
Next we present some numerical results for the order 4 BDF method which
in the formulation \refeq{multstep-def} reads as follows,
$  \frac{1}{12} ( 
25 \myphi_{r+4} \minus 48 \myphi_{r+3} \plus 36 \myphi_{r+2} \minus 16 \myphi_{r+1}  + 3 \myphi_r) =
h \myfun_{r+4} $ for $ r = \myseqq{0}{1}{n-4} $.
This method is applied
to the same operator as for the first numerical experiment but with a different \rhs:
\begin{align}
\ints{0}{x}
{\cos\kla{x-y} u\klasm{y} }{dy}
=
\underbrace{1-\cos x}_{\displaystyle =: f(x)} \for \intervalarg{x}{0}{1},
\label{eq:numeric-b}
\end{align}
with exact solution
$ u\klasm{y} = y $
for $ \intervalarg{y}{0}{1} $.
The conditions
of Assumption \ref{th:msm-assump} are satisfied with $ \emm = \myp = \myqq= 4 $.
\Stepsizes, noise levels, \inapps and starting values are chosen similar to the example considered above. The results are shown in Table \ref{tab:num2}. 
\begin{table}[h!]
\hfill
\begin{tabular}{|| r | c |@{\hspace{5mm} } l | c | c ||}
 \hline
 \hline
$ N $  
& $ \delta $
& $ 100 \myast \delta/\maxnorm{f} $
& $ \ \max_{n} \modul{\undelta - u\klasm{\xn}} \ $
& $ \ \max_{n} \modul{\undelta - u\klasm{\xn}} \ / \delta^{4/5} \ $
\\ \hline \hline
$ 32$ & $3.0 \myast 10^{-8}$ & $6.48 \myast 10^{-6}$ & $7.14 \myast 10^{-6}$ & $7.48$ \\
$  64$ & $9.3 \myast 10^{-10}$ & $2.03 \myast 10^{-7}$ & $4.85 \myast 10^{-7}$ & $8.14$ \\
$ 128$ & $2.9 \myast 10^{-11}$ & $6.33 \myast 10^{-9}$ & $2.85 \myast 10^{-8}$ & $7.65$ \\
$ 256$ & $9.1 \myast 10^{-13}$ & $1.98 \myast 10^{-10}$ & $2.11 \myast 10^{-9}$ & $9.07$ \\
$ 512$ & $2.8 \myast 10^{-14}$ & $6.18 \myast 10^{-12}$ & $1.28 \myast 10^{-10}$ & $8.83$ \\
$ 1024$ & $8.9 \myast 10^{-16}$ & $1.93 \myast 10^{-13}$ & $2.32 \myast 10^{-11}$ & $25.50$ \hspace{1mm}  \\
\hline
\hline
\end{tabular}
\hfill 
\caption{Numerical results of the 4th order BDF method applied to equation \refeq{numeric-b}}
\label{tab:num2}
\end{table} 
\end{example}
\begin{example}
\label{th:num_example_3}
Next we present the results of numerical experiments with the second order Adams--Bashfort method
$ \myphi_{r+2} - \myphi_{r+1} = \tfrac{h}{2}(3 \myfun_{r+1} -\myfun_r) $ for $ r = \myseqq{0}{1}{n-2} $.
The quadrature scheme formulation of this method, see \refeq{msm-noise-proofrep}, is
$ \myphi_{n} 
= \tfrac{h}{2}(3 \myfun_{n-1} +2 \myfun_{n-2} + \cdots + 2 \myfun_{1} - \myfun_{0}) + \myphi_{1}
= \tfrac{h}{2}(3 \myfun_{n-1} +2 \myfun_{n-2} + \cdots + 2 \myfun_{2} + 3 \myfun_{1}) $, 
where the latter identity follows from  
the fact that $ \wstart{10} =
\wstart{11} = \frac{1}{2} $, see \refeq{start-weights-system-matrix}.

This method is applied to 
the following test problem: 
\begin{align}
\ints{0}{x}
{\kla{1+x-y} u\klasm{y} }{dy}
=
\underbrace{x-1+e^{-x}}_{\displaystyle =: f(x)} \for \intervalarg{x}{0}{1},
\label{eq:numeric-c}
\end{align}
with exact solution
$ u\klasm{y} = ye^{-y} $
for $ \intervalarg{y}{0}{1} $.
The conditions
of Assumption \ref{th:msm-assump} are satisfied with $ \emm = \myp = \myqq = 2 $.
\Stepsizes, noise levels, \inapps and starting values are chosen similar to the example considered above. The results are shown in Table \ref{tab:num3}. 
\begin{table}[h!]
\hfill
\begin{tabular}{|| r | c |@{\hspace{5mm} } l | c | c ||}
 \hline
 \hline
$ N $  
& $ \delta $
& $ 100 \myast \delta/\maxnorm{f} $
& $ \ \max_{n} \modul{\undelta - u\klasm{\xn}} \ $
& $ \ \max_{n} \modul{\undelta - u\klasm{\xn}} \ / \delta^{2/3} \ $
\\ \hline \hline
$  32 $ & $3.1 \myast 10^{-5}$ & $8.76 \myast 10^{-3}$ & $1.93 \myast 10^{-3}$ & $1.98$ \\
$   64 $ & $3.8 \myast 10^{-6}$ & $1.07 \myast 10^{-3}$ & $5.21 \myast 10^{-4}$ & $2.13$ \\
$  128 $ & $4.8 \myast 10^{-7}$ & $1.31 \myast 10^{-4}$ & $1.29 \myast 10^{-4}$ & $2.11$ \\
$  256 $ & $6.0 \myast 10^{-8}$ & $1.63 \myast 10^{-5}$ & $3.84 \myast 10^{-5}$ & $2.52$ \\
$  512 $ & $7.5 \myast 10^{-9}$ & $2.03 \myast 10^{-6}$ & $8.99 \myast 10^{-6}$ & $2.36$ \\
$ 1024$ & $9.3 \myast 10^{-10}$ & $2.54 \myast 10^{-7}$ & $2.36 \myast 10^{-6}$ & $2.47$ \\
$ 2048$ & $1.2 \myast 10^{-10}$ & $3.17 \myast 10^{-8}$ & $5.95 \myast 10^{-7}$ & $2.50$ \\
$ 4096$ & $1.5 \myast 10^{-11}$ & $3.96 \myast 10^{-9}$ & $1.60 \myast 10^{-7}$ & $2.68$ \\
\hline
\hline
\end{tabular}
\hfill 
\caption{Numerical results of the 2nd order Adams--Bashfort method applied to equation \refeq{numeric-c}}
\label{tab:num3}
\end{table} 

\noindent
Note that the relative errors in the \rhs 
presented in the third column (of all three tables in fact)
are rather small, respectively.
\end{example}
\begin{example}
\label{th:num_example_4}
Here we consider again the second order Adams--Bashfort method,
see Example \ref{th:num_example_3}, this time 
applied to the problem of numerical differentiation:
\begin{align}
\ints{0}{x}{
u\klasm{y} }{dy}
= f(x) \for \intervalarg{x}{0}{1},
\quad \textup{with} \
u(y) = \left\{\begin{array}{rl} 
2y, & 0 \le y \le \frac{1}{2}, \\
2(1-y), & \frac{1}{2} < y \le 1,
\end{array} \right.
\label{eq:numeric-d}
\end{align}
which means $ u \in \UCpl{0}\interval{0}{1} $ in fact.
We consider the balancing principle, and for this
we need to take a closer look at Proposition \ref{th:msm-estimate-with-constant}.
Elementary computations show that
$ \maxnorm{T^{-1}} = \tfrac{5}{2} $ and
$ \mysumtxt{s=0}{\infty} \vert \gaminv_{s} \vert = \frac{4}{3} $. 
This shows that estimate \refeq{msm-estimate-with-constant}
holds with $ C_2 = \frac{19}{3} $, and
thus we may choose $ \myeta = 13.0 $ in \refeq{hdelta-def}.
For each considered noise level $ \delta $, the integers $ \squer $ and
$\Nsmin $ are chosen such that
$\hs[0] $ is the largest \stepsize $ \le \delta^{1/2} $, and
$\hs[\squer] $ is the smallest \stepsize satisfying $ \ge \delta^{1/3} $
(see \refeq{hnull-gross-hsquer-klein}).
We choose $ \kappa = 1 $ in \refeq{hs-def}.
The results of the numerical experiments are shown in Table \ref{tab:num4}. 
\begin{table}[h!]
\hfill
\begin{tabular}{|| r | c |@{\hspace{5mm} } r |  c | c | c ||}
 \hline
 \hline
$ \delta $ \hspace{4mm} \mbox{}
& $ 100 \myast \delta/\maxnorm{f} $
& $ N(\delta) $  
& $ h(\delta)/\delta^{1/2} $  
& $ \ \max_{n} \modul{\enndelta[\n]} \ $
& $ \ \max_{n} \modul{\enndelta[\n]} \ / \delta^{1/2} \ $
\\ \hline \hline
 $1.0 \myast 10^{-5}$ & $2.00 \myast 10^{-3}$ & $  92$ & $ 3.44$ & $1.45 \myast 10^{-2}$  & 4.60 \\
 $2.5 \myast 10^{-6}$ & $5.00 \myast 10^{-4}$ & $ 146$ & $ 4.33$ & $9.10 \myast 10^{-3}$  & 5.76 \\
 $6.2 \myast 10^{-7}$ & $1.25 \myast 10^{-4}$ & $ 232$ & $ 5.45$ & $5.77 \myast 10^{-3}$  & 7.30 \\
$1.6 \myast 10^{-7}$ & $3.13 \myast 10^{-5}$ & $ 740$ & $ 3.42$ & $1.80 \myast 10^{-3}$  & 4.55 \\
 $3.9 \myast 10^{-8}$ & $7.81 \myast 10^{-6}$ & $1176$ & $ 4.30$ & $1.15 \myast 10^{-3}$  & 5.83 \\
\hline
\end{tabular}
\hfill 
\caption{Numerical results of the 2nd order Adams--Bashfort method, applied to equation \refeq{numeric-d}}
\label{tab:num4}
\end{table} 
\end{example}
\section{Conclusions}
In the present paper we consider the regularization of linear first-kind Volterra integral equations with smooth kernels and perturbed given \rhss.
As regularization scheme we consider quadrature methods that are generated by
linear multistep methods for solving ODEs, with an appropriate starting procedure.
%%% and with the step size as a regularization parameter. 
%%%Both a priori as well a posteriori choices of the step size are considered.
The regularizing properties of an a priori choice of the \stepsize as well as the balancing principle as an adaptive choice of the \stepsize are analyzed, with a variant of the balancing principle which sometimes requires less amount of computational work than the standard version of this principle.

In the case of exact data, the considered scheme 
is similar to that in Wolkenfelt (\mynocitea{Wolkenfelt}{79}, \mynocitea{Wolkenfelt}{81}).
However, our analysis is different from that in those two papers and 
allows less smoothness of the involved functions in fact.
All used smoothness assumptions in the present paper are of the form 
$ \UCpl{\myqq-1} $ instead of $ \Cp{\myqq} $ which enlarge the classes of admissible functions further.

It turns out that an application of the balancing principle for the choice of the \stepsize is possible, but for general kernels $ k $ the coefficient of the error propagation term $ \lfrac{\delta}{h} $ turns out to be rather large which in fact results from an application of the discrete Gronwall inequality in the proof of 
Theorem~\ref{th:main-msm}.

%
%%\bibliography{../datenbanken/standard,../datenbanken/volterra,../datenbanken/numa,../datenbanken/illposed,../datenbanken/fuan}

\end{document}